\documentclass[10pt]{article}
\usepackage[utf8]{inputenc}
\usepackage[english]{babel}
\usepackage{comment}

\usepackage[dvipsnames]{xcolor}

\definecolor{myorange}{RGB}{255, 176, 1}
\definecolor{mygreen}{RGB}{55, 184, 78}
\definecolor{mycyan}{RGB}{51, 255, 255}

\usepackage{amsfonts}
\usepackage{amsthm,amsmath,amssymb}

\usepackage{authblk}

\newcommand{\RR}{{\mathbb{R}}}

\newcommand{\der}{\partial}

\newcommand{\om}{\Omega}
\newcommand{\oms}{\Omega^{\sharp}}
\newcommand{\Gs}{G^{\sharp}}
\newcommand{\fs}{f^{\sharp}}
\newcommand{\gs}{g^{\sharp}}
\newcommand{\Kd}{\mathcal{K}^{\diamond}}
\newcommand{\dive}{\text{\normalfont div}}

\usepackage{amsthm}

\newtheorem{theorem}{Theorem}[section]
\newtheorem{lemma}[theorem]{Lemma}
\newtheorem{proposition}[theorem]{Proposition}

\newtheorem{remark}[theorem]{Remark}
\newtheorem{definition}[theorem]{Definition}

\numberwithin{equation}{section}

\date{}

\usepackage{graphicx}
\usepackage{caption}
\usepackage{subcaption}

\begin{document}
\title{ Lipschitz stable determination of polyhedral conductivity inclusions from local boundary measurements}
\author[1]{Aspri Andrea} 
\author[2]{Beretta Elena}
\author[3]{Francini Elisa}
\author[3]{Vessella Sergio}

\affil[1]{Department of Mathematics, Università degli Studi di Pavia}
\affil[2]{Department of Mathematics, NYU Abu Dhabi}
\affil[3]{Department of Mathematics and Computer Science, Università degli Studi di Firenze}

\maketitle

\begin{abstract}
We consider the problem of determining a polyhedral conductivity inclusion embedded in a homogeneous isotropic medium from boundary measurements. We prove global Lipschitz stability for the polyhedral inclusion from the local Dirichlet-to-Neumann map extending in a highly nontrivial way the results obtained in \cite{BerFra20} and  \cite{BerFraVes21} in the two-dimensional case to the three-dimensional setting.
\end{abstract}

\section{Introduction}
In this paper we analyze the nonlinear inverse problem of determining a polyhedron embedded in a three-dimensional homogeneous isotropic conducting body from boundary measurements.
More precisely,  we consider the conductivity equation 
\begin{equation}\label{conductivity}
             \textrm{ div }(\gamma_D\nabla u)  =  0\mbox{ in }\om\subset\mathbb{R}^3, \\
\end{equation}
where 
\[\gamma_D=1+(k-1)\chi_{D},\]  
{with $D$ a polyhedral inclusion} strictly contained in a bounded domain $\Omega$, and $k\neq 1$ is a given, positive constant.

This class of conductivity inclusions appears in applications, like for example in geophysics exploration, where the medium (Earth) under inspection contains heterogeneities in the form of rough bounded subregions (for example subsurface salt or limestone bodies) with different conductivity properties \cite{Zhdanov1994TheGM}.

We establish a Lipschitz stability estimate for the Hausdorff distance of polyhedral conductivity inclusions in terms of the local Dirichlet-to-Neumann (DtN) map, and, as a byproduct, a uniqueness result which is new in this general setting.
An analogous, though less general result  was obtained in \cite{BerdeHFrVes} in the case of the Helmholtz equation.
{We would like to point out that in principle it should be possible to recover in a Lipschitz stable way both the polyhedral inclusion and the constant conductivity from boundary data but in order to reduce the technical complexity of the proof we decided to treat the case where the conductivity is fixed.}

Lipschitz stability estimates are of key importance in {practical applications}. In fact, they provide a useful framework for optimization when using iterative methods, see for example \cite{deHoopetal,AS2021} so that the recovery of polyhedral interfaces becomes a shape optimization problem, see \cite{BerMicPerSan18,Shi} for the reconstruction of polygonal and polyhedral inclusions.

There is a wide literature on Lipschitz stability for the inverse conductivity problem when unknown coefficients depend on finitely many parameters and infinitely many measurements are available, see for example \cite{AleVes05}, \cite{BerFra11},\cite{AldeHG}, \cite{AldeHGS}, \cite{GS}, \cite{FGS} and \cite{BerFra20,BerFraVes21} while in the case of finitely many measurements we refer to \cite{ABV,BacV} and to the more recent work \cite{AS2021,Alberti2020,H1,H2}.

To our knowledge uniqueness and stability for general polyhedral conductivity inclusions from finitely many measurements {are} an open issue. 
Unique determination from one suitably chosen measurement has been proved in \cite{BFK} restricting to the class of convex polyhedra. Logarithmic stability from one measurement has been derived in \cite{LiuTsou} in the two-dimensional case for polygonal conductivity inclusions {and in \cite{LiuTsoYan21} some preliminary results are obtained for the determination of a class of smooth two-dimensional inclusions}.

Also, we would like to mention that the results obtained
recently in \cite{Alberti2020} in an abstract setting and where Lipschitz continuity from finitely many measurements has been proved if the unknown belongs to a suitable finite dimensional nonlinear manifold seem not to include the case of polygonal and polyhedral conductivity inclusions. 

On the other hand, in several
applications, like the geophysical one, many measurements are at disposal on some part of the boundary, justifying the use of the local Dirichlet-to-Neumann map \cite{BCFLM}.

We would like to emphasize that the result we obtain is not at all a straightforward extension of the two-dimensional results obtained previously in \cite{BerFra20} and \cite{BerFraVes21} since it requires to deal with the more complex three-dimensional geometric setting. In fact, our main result relies on some preliminary rather technical but crucial geometric properties on admissible polyhedra $D\in \mathfrak{D}$ satisfying minimal a priori assumptions of Lipschitz type. In particular, for two polyhedra in $\mathfrak{D}$ we are able to compare the Hausdorff distance of their boundaries and a modified distance defined in Section \ref{sec3}, Definition \ref{distmod}. These properties are then used to derive a first rough stability estimate of logarithmic type relating the Hausdorff distance between the boundaries of the polyhedra and the corresponding DtN maps.
The stability estimate is obtained along the lines proposed in \cite{ADC} and \cite{ADCMR}: computing the difference of the local DtN 
along a pair of singular solutions for the conductivity
operator with singularities $y, z \in  \mathbb{R}^3\setminus \Omega$  close to 
$\partial \Omega$ exploiting unique continuation and regularity properties of this function, denoted by $S(y, z)$, and finally coupling upper and lower bounds of $S(y, z)$.

Furthermore, as in \cite{BerFraVes21}, a crucial step to establish our Lipschitz stability is to prove smoothness of the local DtN map and to establish a lower bound of the directional derivative of the local DtN map. We construct an ad-hoc Lipschitz vector field, use a distributed representation formula of the derivative, derived in \cite{BerMicPerSan18}, and integrate by parts far from edges and vertices taking advantage of regularity properties of solutions to (\ref{conductivity}) close to smooth interfaces and avoiding the complex singular behaviour solutions to (\ref{conductivity})  exhibit close to vertices and edges.    
Finally, collecting the results of Sections \ref{sec4} and \ref{sec5} in Section \ref{sec6} we prove our main result.

It would be interesting to extend
the results of stability to the more general geometric configuration where the reference domain $\Omega$ is in the form of an inhomogeneous layered medium. 
This kind of geometrical setting originates from applications, for example, in {geophysical exploration}, where the medium under inspection (for example the Earth) is layered and contains heterogeneities in the form of rough bounded sub-regions  with different conductivity properties, \cite{F}.
Moreover, the theoretical results in this paper contain the building blocks towards successful numerical reconstruction procedures based on, {for example}, shape derivative and level set techniques, as in \cite{AlbLauStu20,FepAllBorCorDap19,Lau18,Lau20,LauMef16,LauStu16}.

The plan of the paper is the following: In Section \ref{sec2}, we list the main a priori assumptions on the reference medium, the admissible polyhedral inclusions $D\in\mathfrak{D}$, the conductivity parameter and the data and state our main result, Theorem \ref{mainteo}.
In Section \ref{sec3}, we collect and prove the main geometric properties on polyhedra belonging to the class $\mathfrak{D}$ that are crucial to derive our main stability result.
In Theorem \ref{stablog} of Section \ref{sec4}, we derive a first rough logarithmic stability estimate.  
In Section \ref{sec5}, we analyse the differentiability properties of the local DtN map, establish a formula for the directional derivative, prove its continuity and derive a lower bound (Proposition \ref{prop:lower_bound_F'}).   
Finally, in Section \ref{sec6}, collecting the results of Section \ref{sec4} and \ref{sec5}, we prove our main stability result (Theorem \ref{mainteo}). 
The appendix collects some technical proofs.

\section*{Notation}
We begin by setting notation that we will use throughout and recalling some of the needed definitions.\\
Given $P\in\mathbb{R}^3$, and $R>0$, we denote by $B_{R}(P)$ the ball of center $P$ and radius $R$, that is 
\begin{equation}\label{def:ball}
B_{R}(P):=\{x\in\mathbb{R}^3:\ |x-P|<R\},
\end{equation}
and by $B'_{R}(P)$ a disc centered at $P$ with radius $R$, contained in a specific plane, which will be specified each time. We omit $P$ when the center of the ball is in the origin.

We utilize standard notation for inner products, that is $x\cdot y=\sum_i x_i y_i$. 
Given $A$ and $B$ bounded sets in $\mathbb{R}^3$, we recall that 
\begin{equation}
    dist(x,A)=\inf\{|x-a|:\ a\in A\},\quad \textrm{and}\quad dist(A,B)=\inf\{|a-b|:\ a\in A,\ b\in B\},  
\end{equation}
and we define the Hausdorff distance between two bounded and closed sets $C$ and $D$ in $\mathbb{R}^3$ as
\begin{equation}\label{eq:hausdorff distance}
    d_H(C,D)=\max\left\{\max_{x\in C}dist(x,D),\max_{x\in D}dist(x,C)\right\}.
\end{equation}
With $Int(C)$ we denote the set of interior points of $C$.
Given two closed simply connected and bounded flat surfaces $F_1$ and $F_2$ contained in $\mathbb{R}^3$, and assuming that $F_1\cap F_2=:\sigma$, where $\sigma$ is a segment and such that $\sigma\neq \emptyset$, then we denote by  
$Int_{\RR^2}\left(F_1\right)$ and  $Int_{\RR}\left(\sigma\right)$ the interior of the set relative to the plane and the line that contain $F_1$ and $\sigma$, respectively.

\section{Assumptions and main result}\label{sec2}
Let us start setting up the definition of {a polyhedron}, the notation for faces and vertices of the polyhedron and the a-priori assumptions that are needed in order to derive our main result. 
\begin{definition}\label{def1}
A closed subset $D\subset \RR^3$ is a polyhedron if:
\begin{equation}
    D \mbox{ is homeomorphic to a ball in } \RR^3;
\end{equation}
the boundary $\der D$
is given by
\begin{equation}
\der D=\bigcup_{j=1}^H F^D_j
\end{equation}
where each $F_j^D$ is a closed simply connected plane polygon (that is called a face of $D$) and
\begin{equation}
    Int_{\RR^2}\left(F_i^D\right)\cap Int_{\RR^2}\left(F_j^D\right)=\emptyset \mbox{ for }i\neq j.
\end{equation}
For $i\neq j$, $\sigma^{D}_{ij}=F_i^D\cap F_j^D$ is called an edge of $D$ if $Int_{\RR}\left(\sigma^D_{ij}\right)\neq \emptyset$.
The non empty intersection of two edges is called a vertex $V^D$ of $D$.

\end{definition}

\subsection{Assumptions on the polyhedral inclusion and on the reference medium}\label{sec:assumptions}
We consider a class of non degenerate polyhedra: let 
\[ r_0,\quad R_0,\quad \theta_0,\quad M_0\]
be given positive numbers such that $\theta_0\in(0,\pi/2)$ and $r_0< R_0$. 

Let  $\om\subset\RR^3$ be a bounded domain such that
\begin{equation}\label{diam}
    diam\left(\om\right)\leq R_0,
\end{equation}
where $diam(\om)$ denotes the diameter of $\Omega$.
%

We say that a polyhedron $D\subset\Omega$ is in $\mathfrak{D}=\mathfrak{D}(r_0, R_0,\theta_0,M_0)$ if the following assumptions hold.
\begin{description}
\item[Strict Inclusion:] 
\begin{equation}\label{distfron}
    dist(D,\der\om)\geq r_0.
\end{equation}
\item[Dihedral angle non-degeneracy:] at each edge of $D$ the angle between the intersecting faces has width $\alpha$ such that 
\begin{equation}\label{angolifacce}
    \alpha\in(\theta_0,\pi-\theta_0)\cup(\pi+\theta_0,2\pi-\theta_0).
\end{equation}
\item[Face non-degeneracy:] for any polygonal face $F^D$ there exists $x_0\in F^D$ such that
\begin{equation}\label{discofacce}
    B^\prime_{r_0}(x_0)\subset F^D,
\end{equation}
where $B^\prime_{r_0}(x_0)$ is contained in the plane containing $F^D$.
\item[Edge non-degeneracy:] for each edge $\sigma_{ij}^D$ of $D$
\begin{equation}\label{lunghlati}
    length\left(\sigma_{ij}^D\right)\geq r_0.
\end{equation}
\item[Face angle non-degeneracy:] each internal angle $\beta$ of each face $F^D$ satisfies
\begin{equation}\label{angolinterni}
    \beta\in(\theta_0,\pi-\theta_0)\cup(\pi+\theta_0,2\pi-\theta_0).
\end{equation}
\item[Lipschitz regularity]
\begin{equation}\label{lip}
 \om\setminus D   \mbox{ is connected and has Lipschitz boundary with constants }r_0\mbox{ and }M_0, 
\end{equation}that is: for every $P\in\der(\om\backslash D)$ there is a rigid transformation of coordinates under which $P\equiv0$ and 
\begin{equation*}
    \left(\om\setminus D\right)\cap R_{M_0,r_0}=\left\{(x_1,x_2,x_3)\,:\,\Psi(x_1,x_2)<x_3\right\}
\end{equation*}
where 
\begin{equation*}
    R_{M_0,r_0}=[-r_0,r_0]^2\times [-2M_0r_0,2M_0r_0]
\end{equation*}
and $\Psi:[-r_0,r_0]^2\to \RR$ is such that
$\Psi(0,0)=0$ and
\begin{equation*}
    \left|\Psi(x_1,x_2)-\Psi(x^\prime_1,x^\prime_2)\right|\leq M_0\sqrt{(x_1-x^\prime_1)^2+(x_2-x^\prime_2)^2},
\end{equation*}
for every $x_1$, $x_2$, $x_1^\prime$, $x_2^\prime\in[-r_0,r_0]$.
\end{description}
\begin{remark}\label{rem1}
The number of vertices $V^D$, edges $\sigma_{ij}^D$ and faces $F_j^D$ of a polyhedron in $\mathfrak{D}$ is bounded from above by a constant $N_0$ depending only on $r_0, R_0$, and $M_0$.
\end{remark}
\begin{remark}\label{rem2}
Recall that \eqref{lip} is not implied by the previous assumptions. Figure \ref{nonlip} shows a polyhedron satisfying \eqref{angolifacce} -- \eqref{angolinterni} but not \eqref{lip} at $P$.
{\begin{remark}
Some of the previous assumptions are technical and instrumental to derive some of the proofs. It might be possible, in principle, that using other techniques these assumptions can be relaxed.  
\end{remark}
}
\begin{figure}[!h]
	\centering
	\includegraphics[scale=0.5]{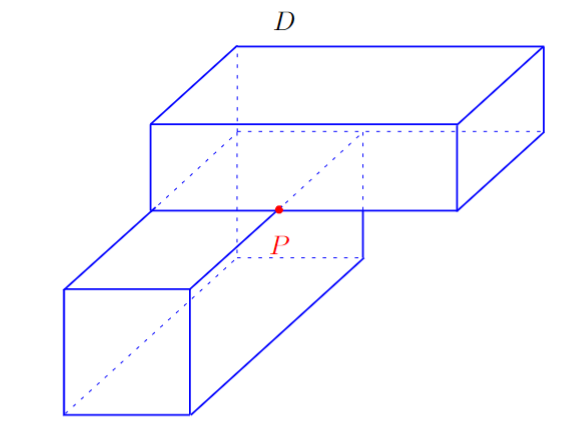}
	\caption{An example of a polyhedron satisfying \eqref{angolifacce} -- \eqref{angolinterni} but not \eqref{lip} at $P$. {We refer the reader to \cite[Example 3.7]{ErnGuer21} for a detailed explanation, by using the uniform cone property, of the fact that $D$ is not a Lipschitz domain.}}
	\label{nonlip}
\end{figure}
\end{remark}

Let 
\begin{equation}\label{eq:conductivity_coeff}
\gamma_D:=1+(k-1)\chi_D,    
\end{equation}
where $\chi_D$ is the characteristic function of $D\in\mathfrak{D}$, and $k$ is a positive constant such that 
{
\begin{equation}\label{contrast}
   \min(k,|k-1|)\geq \kappa_0.
\end{equation}
}
Finally let us state the assumptions on the part of the boundary on which we measure our data.
Let $\Sigma$ be an open portion of  $\partial\om$ with size at least $r_0$, i.e.  we assume there exists at least one point $P_{\Sigma}\in \Sigma$ such that
\begin{equation}\label{asssigma}
    dist(P_{\Sigma},\partial\om \setminus\Sigma)\geq r_0.
\end{equation}

In the sequel, we will refer to the set of parameters
\[ r_0,\quad R_0,\quad \theta_0,\quad M_0,\quad \kappa_0\]
as the {\it a priori data}.

\subsection{The local Dirichlet to Neumann map}
We define  
\begin{equation*}
    H^{\frac{1}{2}}_{co}(\Sigma):=\Bigg\{\varphi\in H^{\frac{1}{2}}(\partial\om)\, :\, \textrm{supp}\ \varphi\subset \Sigma \Bigg\},
\end{equation*}
and with $H^{-\frac{1}{2}}_{co}(\Sigma)$ its topological dual. \\
 Given $f\in H^{\frac{1}{2}}_{co}(\Sigma)$, we consider the boundary value problem
\begin{equation}\label{eq:main_prob}
    \begin{cases}
    \dive(\gamma_D\nabla u)=0 & \textrm{in}\ \om \\
    u=f & \textrm{on}\ \partial\om.
    \end{cases}
\end{equation}
Let us denote by $\Lambda_{\gamma_D}^{\Sigma}$ the local DtN map, that is the map 
\begin{equation}\label{eq:DtN map}
\begin{aligned}
    \Lambda_{\gamma_D}^{\Sigma}: H^{\frac{1}{2}}_{co}(\Sigma) &\to H^{-\frac{1}{2}}_{co}(\Sigma)\\
    f & \to \frac{\partial u}{\partial \nu}\bigg\lfloor_{\Sigma}
\end{aligned}    
\end{equation}
where $u\in H^1(\om)$ is the solution to \eqref{eq:main_prob}, and $\nu$ is the outer unit normal vector to $\partial\Omega$. 
The norm of the local DtN map in the space of linear operators $\mathcal{L}\left(H^{\frac{1}{2}}_{co}(\Sigma),H^{-\frac{1}{2}}_{co}(\Sigma)\right)$ is defined by 
\begin{equation*}
    \|\Lambda_{\gamma_D}^{\Sigma}\|_{\star}:=\textrm{sup}\bigg\{ \|\Lambda_{\gamma_D}^{\Sigma}\varphi\|_{H^{-\frac{1}{2}}_{co}(\Sigma)}/\|\varphi\|_{H^{\frac{1}{2}}_{co}(\Sigma)}\, :\, \varphi\neq 0 \bigg\}.
\end{equation*}
As in \cite{AleKim12} the DtN map can be defined as the operator characterized by 
\begin{equation*}
    \big\langle \Lambda_{\gamma_D}^{\Sigma}f, \phi \big\rangle =\int_{\om}\gamma_D \nabla u\cdot \nabla v\, dx,
\end{equation*}
for all $\phi,f\in H^{\frac{1}{2}}_{co}(\Sigma)$, where $u\in H^1(\om)$ is the solution to \eqref{eq:main_prob}, and $v$ is any $H^1(\om)$-function such that $v\lfloor_{\Sigma}=\phi$.\\
\subsection{The main result}
Recalling the definition of the Hausdorff distance, see \eqref{eq:hausdorff distance}, we state here our main Lipschitz stability result: 
\begin{theorem}\label{mainteo} Let $\om$ be a bounded domain with Lipschitz boundary satisfying \eqref{diam},  let $D_0$ and $D_1\in \mathfrak{D}$ (that is they satisfy assumptions \eqref{distfron}-\eqref{lip}), let $k$ satisfy  \eqref{contrast} and let $\Sigma$ be an open portion of $\partial\om$ satisfying \eqref{asssigma}. 
Then,  
there exists  $C$ depending only on the a priori data such that
\begin{equation}\label{eq:mainteo}
    d_{H}\left(\der D_0,\der D_1\right)\leq C
\left\|\Lambda_{\gamma_{D_0}}^{\Sigma}-\Lambda_{\gamma_{D_1}}^{\Sigma}\right\|_*,
\end{equation}
where
\[\gamma_{D_i}=1+(k-1)\chi_{{D_i}}\quad\textrm{for}\quad i=0,1.\]
\end{theorem}

The proof of Theorem \ref{mainteo} is postponed at Section \ref{sec6}, after proving some intermediate results based essentially on the following steps and strategy:

\begin{enumerate}
   \item  Prove that it always exists a suitable tubular neighborhood connecting any point on $\partial\Omega$ to a special interior point of the face of one of the two polyhedra $D_0$ or $D_1$, without crossing $D_0\cup D_1$. To this aim, we introduce a specific distance (called ``modified distance'') (see Definition \ref{distmod}) and exploit 
    the connection between the modified and the Hausdorff distances (see Proposition \ref{prop2}).
    \item The results of the previous point allow us to establish a rough (logarithmic) stability estimate of the Hausdorff distance between $D_0$ and $D_1$ in terms of the difference between the corresponding DtN map, see Theorem \ref{stablog}. This is obtained by propagating the smallness of data from $\Sigma$ along the tubular neighborhood.
    \item The logarithmic stability estimate implies that if two DtN maps are close enough, the two polyhedra have the same number of vertices, faces and edges, see Proposition \ref{distvert}. When this happens, it is possible to define a regular  vector field that transforms $D_0$ into $D_1$.  
    We also prove smoothness of the local DtN map and establish a lower bound of its derivative with respect to the movement of the polyhedron
    \item The regularity of the DtN map and the lower bound allow us to improve the stability estimates and to get Theorem \ref{mainteo}.
\end{enumerate}

\section{Some useful geometric results on polyhedra}\label{sec3}
In this section we collect some geometric results on polyhedra in the class $\mathfrak{D}$. We first establish the relation between the Hausdorff distance of two polyhedra in $\mathfrak{D}$ and the Hausdorff distance of their boundaries, see Proposition \ref{prop1}. Afterwards, we consider a modified distance between two polyhedra (see Definition \ref{distmod}) that was introduced in \cite{AS,ADCMR} and establish an upper bound of the Hausdorff distance of the boundaries of two polyhedra in terms of their modified distance, see Proposition \ref{prop2}. This last
property together with the main result of this section, that is Proposition \ref{lemmageometrico}, will be crucial in Section \ref{sec4} to establish our first logarithmic stability estimate.

Proposition \ref{lemmageometrico} here corresponds to Lemma 4.2 in \cite{ADCMR} where it is stated under the assumption of inclusions with $C^{1,\alpha}$ boundaries; this regularity assumption allows to show that the union of two such inclusions has Lipschitz boundary. Unfortunately, this is not the case for polyhedra in $\mathfrak{D}$. For this reason, in order to prove Proposition \ref{lemmageometrico}, we have to rely on a fine result from \cite{R08} stating that if two polyhedra in $\mathfrak{D}$ are close enough, in some neighborhood of some special point in the interior of one of the faces, the boundaries of the two polyhedra are relative graphs of affine functions (see Proposition \ref{prop4}).

The last key geometric result, contained in Proposition \ref{distvert}, states that if two polyhedra in $\mathfrak{D}$ are close enough, then they have the same number of vertices, edges and faces.

\subsection{Metric results}
In this subsection we use some results from \cite{R08}.
For this, we observe that our class of polyhedra $\mathfrak{D}$ is a subset of the class of polyhedra $\mathcal{A}_{p,0}(h)$ (defined in \cite{R08}) for some $h>0$ depending only on the a priori data. 

Let us set some useful notation.
 Given $P\in\RR^3$, a direction $\nu\in\RR^3$, $l>0$ and $\vartheta\in (0,\pi/2)$, we denote by
\begin{equation}\label{cone}
    \mathcal{C}(P,\nu,l,\vartheta)=\left\{x\in\RR^3\,:\,(x-P)\cdot\nu\geq|x-\overline{x}|\cos \vartheta,\, |x-P|\leq l\right\}
\end{equation}
the closed cone with vertex $P$, axis  $\nu$, width $\vartheta$, and apothem $l$.

\begin{remark}\label{remcono}
By assumption \eqref{lip}, for each $P\in \der\left(\om\setminus D\right)$ there exist a direction $\nu$, a positive $l$ and $\vartheta\in(0,\pi/2)$ depending only on the a priori data, such that 
\begin{equation*}
    \mathcal{C}(P,\nu,l,\vartheta)\subset\overline{(\om\setminus D)}
\end{equation*}
and, if $P\in\der D$
\begin{equation*}
    \mathcal{C}(P,-\nu,l,\vartheta)\subset D.
\end{equation*}
\end{remark}

The proposition below (that corresponds to Proposition 2.4 in \cite{R08} to which we refer for the proof) establishes the equivalence in $\mathfrak{D}$ between $d_H(D_0,D_1)$ and $d_H(\der D_0,\der D_1)$.
\begin{proposition}\label{prop1}
Let $D_0$ and $D_1\in \mathfrak{D}$, then there is a positive constant $C_1>1 $ depending on the a priori data only such that 
\begin{equation}\label{12}
    C_1^{-1}d_H(\der D_0,\der D_1)\leq d_H(D_0,D_1)\leq 
    C_1 d_H(\der D_0,\der D_1).
\end{equation}
\end{proposition}
For $D_0$ and $D_1\in\mathfrak{D}$, let $\mathcal{G}$ be the connected component of $\Omega\setminus \left(D_0\cup D_1\right)$ which contains $\partial\Omega$, and let 

\begin{equation}\label{eq:omD}
    \om_{\mathcal{G}}=\Omega\setminus \mathcal{G}.
\end{equation}
Since the value of $d_H(\der D_0,\der D_1)$ can be attained at some point of $\der D_0\cup\der D_1$ that is not necessarily on $\der \Omega_{\mathcal{G}}$ (see, for example, the configuration in Figure \ref{fig:setting}) and, hence, cannot be reached from $\der \om$ without crossing $\der D_0\cup\der D_1$, 
we introduce a modified distance as was defined in \cite{ADCMR}.
\begin{definition}\label{distmod}
\begin{equation}
    d_\mu(D_0,D_1)=\max\left\{\max_{x\in\der D_0\cap \der\Omega_{\mathcal{G}}}dist(x,D_1),\max_{x\in\der D_1\cap \der\Omega_{\mathcal{G}}}dist(x,D_0)\right\}.
\end{equation}
\end{definition}
\begin{figure}[!h]
     \centering
     \begin{subfigure}[b]{0.2\textwidth}
         \centering
         \includegraphics[width=\textwidth]{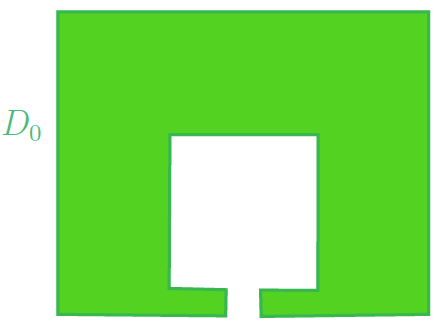}
         \caption{Polyhedron $D_0$ in green.}
         \label{fig:D0}
     \end{subfigure}
    \hspace{3cm}
     \begin{subfigure}[b]{0.2\textwidth}
         \centering
         \includegraphics[width=\textwidth]{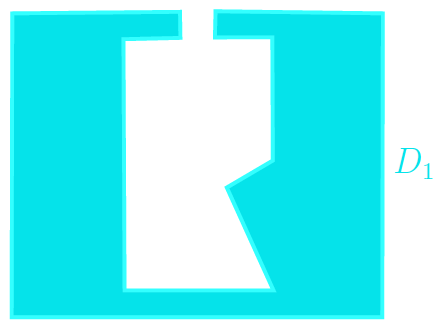}
         \caption{Polyhedron $D_1$ in pale blue.}
         \label{fig:D1}
     \end{subfigure}
      \ \\ \ \\
     \begin{subfigure}[b]{0.3\textwidth}
         \centering
         \includegraphics[width=\textwidth]{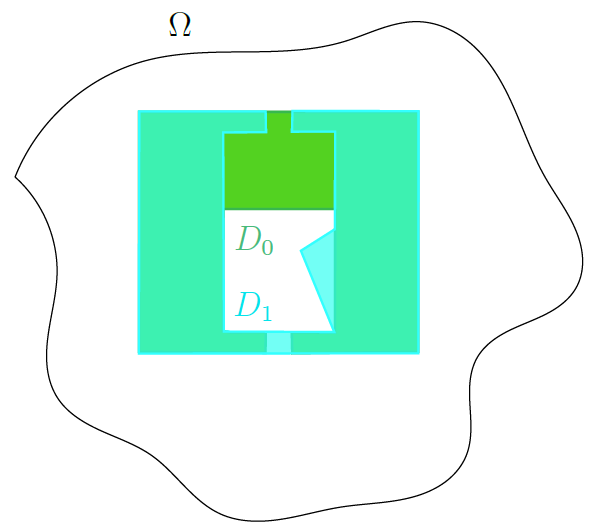}
         \caption{A possible configuration. The two polyhedra are overlapping.}
         \label{fig:configuration}
     \end{subfigure}
     \hspace{2cm}
     \begin{subfigure}[b]{0.25\textwidth}
         \centering
         \includegraphics[width=\textwidth]{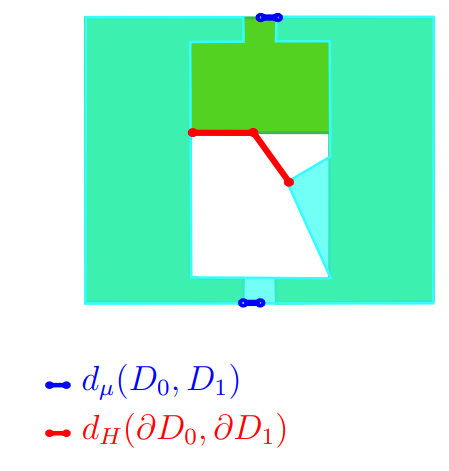}
         \caption{The two distances. For $d_H$ we provide all the points where the longest distance between the two sets occurs.}
         \label{fig:distances}
     \end{subfigure}
        \caption{A 2D-section of a possible geometrical setting. Note that Figure \ref{fig:distances} represents the case when the value of $d_H(\partial D_0,\partial D_1)$ is attained at some point that is not on $\partial\Omega_{\mathcal{G}}$.}
        \label{fig:setting}
\end{figure}
We point out that this is not a metric because, in general, the triangle inequality doesn't hold. It is straightforward to show, see \cite{ADCMR}, that
\begin{equation}\label{14}
    d_\mu(D_0,D_1)\leq d_H(\der D_0,\der D_1).
\end{equation}
In general, $d_{\mu}$  does not bound from above the Hausdorff measure, but, in the class $\mathfrak{D}$ the following result that will be crucial for deriving the stability estimates in Section \ref{sec4}, holds:
\begin{proposition}\label{prop2}
There is a constant $C_2>1$ depending only on the a priori data, such that, for $D_0$, $D_1\in \mathfrak{D}$
\begin{equation*}
    d_H(\der D_0,\der D_1)\leq C_2 d_\mu(D_0,D_1).
\end{equation*}
\end{proposition}
In order to prove Proposition \ref{prop2} we need the following preliminary result:
\begin{lemma}\label{lm3}
Let $D\in\mathfrak{D}$. Then, for every $P\in \der D$ there exists a curve $\mathfrak{c}$ in $\om\setminus D$ connecting $P$ to $\der\om$ such that 
\begin{equation*}
    |z-P|\leq C_2 \, dist(z,D),\quad {\forall z\in\mathfrak{c}},
\end{equation*}
where $C_2>1$ depends only on the a priori data.
\end{lemma}
This lemma corresponds to Proposition 3.3 in \cite{ADC} and to Lemma 4.1 in \cite{ADCMR} for $C^{1,\alpha}$ inclusions. Here, we prove it for $D\in\mathfrak{D}$.

\begin{proof}[Proof of Lemma \ref{lm3}]
By Assumption \eqref{lip} we can apply Lemma 5.5 in \cite{ARRV}, hence, there exists a positive number $a$ depending only on the Lipschitz constant $M_0$  such that the set
\begin{equation}\label{EDt}
    E^D_t=\left\{x\in\overline{\om}\setminus D\,:\, dist(x,\der D)>t\right\}
\end{equation}
is connected for $t\leq a r_0$. 

Let $P\in\der D$; by Remark \ref{remcono}, there exists a cone $\mathcal{C}(P,\nu,l,\vartheta)\subset\overline{(\om\setminus D)}$.  
By easy calculations we can see that, by choosing 
\begin{equation*}
    \tau_0=\frac{l}{1+\sin\vartheta},
\end{equation*}
the point $y_{\tau_0}=P+\tau_0\nu$ satisfies $dist\left(y_{\tau_0}, \partial C(P,\nu,l,\vartheta\right))=\tau_0\sin\vartheta$.
Let us now take $t_0=\min\left\{a r_0,\tau_0\right\}$.
Since $t_0\leq\tau_0$, we have
\begin{equation*}
    dist\left(y_{t_0},\om\setminus D\right)\geq    dist\left(y_{t_0},\partial \mathcal{C}(P,\nu,l,\vartheta)\right)\geq t_0\sin\vartheta,
\end{equation*}
hence $y_{t_0}\in E^D_{t_0\sin\vartheta}$.
Since $t_0\sin\vartheta<ar_0$, then $E^D_{t_0\sin\vartheta}$ is connected.

Let $\mathfrak{c}^\prime$ be a curve in $E^D_{t_0\sin\vartheta}$ that connects $y_{t_0}$ to $\partial\om$ and let $\mathfrak{c}=\mathfrak{c}^\prime\cup [y_{t_0},P]$, where $[y_{t_0},P]$ is the line segment from $P$ to $y_{t_0}$.

If $z\in\mathfrak{c}^\prime$, then $dist(z,\partial D)\geq t_0\sin\vartheta$, hence
\begin{equation*}
|z-P|\leq diam(\om)\leq R_0\leq \frac{R_0}{t_0\sin\vartheta}dist(z,\partial D).
\end{equation*}
If $z\in [y_{t_0},P]$, then $dist(z,\partial D)\geq |z-P|\sin\vartheta$.
In both cases
\begin{equation*}
    |z-P|\leq \max\left\{\frac{R_0}{t_0\sin\vartheta},\frac{1}{\sin\vartheta}\right\},\quad \forall z\in\mathfrak{c}.
\end{equation*}
\end{proof}
We are now ready to prove Proposition \ref{prop2}.
\begin{proof}[Proof of Proposition \ref{prop2}]
Let $P\in \partial D_0$, we have two different cases
\begin{enumerate}
    \item[(i)] $P\in\partial D_0\cap \partial \Omega_{\mathcal{G}}$;
    \item[(ii)] $P\in\partial D_0\setminus \partial \Omega_{\mathcal{G}}$.
\end{enumerate}
In case (i) we have that $P\notin Int(D_1)$, hence
\begin{equation*}
    dist(P,\partial D_1)=dist(P,D_1)\leq d_\mu (D_0,D_1).
\end{equation*}
In case (ii) we have $P\in Int(\Omega_{\mathcal{G}})$. 
By Lemma \ref{lm3}, let $\mathfrak{c}$ be a curve such that 
\begin{equation*}
    \mathfrak{c}\subset \om\setminus D_0
\end{equation*}
connects $P$ to $\partial \om$ and
\begin{equation}\label{1.12}
    |z-P|\leq C_2 dist(z,D_0),\quad\forall z\in \mathfrak{c}.
\end{equation}
Since $P\in Int(\Omega_{\mathcal{G}})$, $\mathfrak{c}$ intersects $\partial\Omega_{\mathcal{G}}$ and, since
\begin{equation*}
    \left(\mathfrak{c}\cap\partial D_0\right)\setminus \{P\}=\emptyset,
\end{equation*}
then
\begin{equation*}
    \left(\mathfrak{c}\cap \Omega_{\mathcal{G}}\right)\cap \partial D_1 \neq\emptyset.
\end{equation*}

Let $\overline{z}\in\mathfrak{c}\cap \partial\Omega_{\mathcal{G}}\cap\partial D_1$. We have
\begin{equation*}
    dist(\overline{z},D_0)\leq \sup_{x\in\partial D_1\cap \partial\Omega_{\mathcal{G}}}dist(x,D_0)\leq d_\mu(D_0,D_1)
\end{equation*}
and, by \eqref{1.12}, we have
\begin{equation}\label{2.12}
    \frac{1}{C_2}|\overline{z}-P|\leq dist(\overline{z},D_0)\leq d_\mu(D_0,D_1).
\end{equation}
Since $\overline{z}\in\partial D_1$ and by \eqref{1.12}
\begin{equation}\label{3.12}
    dist(P,\partial D_1)\leq |\overline{z}-P|,
\end{equation}
hence, from \eqref{2.12} and \eqref{3.12} we have
\begin{equation}\label{1.13}
    dist(P,\partial D_1)\leq C_2 d_\mu(D_0,D_1), \quad \forall P\in \partial D_0,
\end{equation}
and, by symmetry,
\begin{equation}\label{2.13}
    dist(Q,\partial D_0)\leq C_2 d_\mu(D_0,D_1), \quad \forall Q\in \partial D_1.
\end{equation}
Inequalities \eqref{1.13} and \eqref{2.13} imply that
\begin{equation*}
    d_H(\partial D_0,\partial D_1)\leq C_2 d_\mu(D_0,D_1).
\end{equation*}
\end{proof}
\subsection{A useful geometric construction}
The aim of this subsection (see Proposition \ref{lemmageometrico}) is the construction of a special tubular set contained in $\mathcal{G}$ that connects a special point on $\partial\Omega_{\mathcal{G}}$ to any point on $\partial\om$  (and particularly any point on $\Sigma$) and has a fixed positive distance from the rest of the boundaries of the two polyhedra.  In this set we will be able to propagate the information on the DtN map up the the boundary of $\partial \Omega_{\mathcal{G}}$.

In order to construct this tubular set, we need some information on the position of the boundaries of the two polyhedra when they are sufficiently close.   
Proposition \ref{prop4}  below, that is the adaptation to our setting of  Proposition 6.2 in \cite{R08}, states that, in a neighborhood of some point,  the boundaries of the two polyhedra are relative graphs of affine functions that are not too close (see \eqref{distPhi}) .

\begin{proposition}\label{prop4}
 There exist positive constants  $k_1\leq \overline{k}_0\leq k_0$, $K$, $K_1$ and $L_1$ depending only on the a priori data, such that, if $D_0$, $D_1\in\mathfrak{D}$ and 
\begin{equation*}
    d_H(D_0,D_1)\leq k_0 r_0,
\end{equation*}
then there exist $P_0\in\partial D_0$ and $P_1\in\partial D_1$ such that the following conditions are satisfied.
Up to a rigid transformation $P_0=(0,0,0)$, $P_1=(0,0,a_1)$ and
\begin{equation*}
    \partial D_0\cap B_{\overline{k}_0r_0}=\left\{
    (x_1,x_2,x_3)\in  B_{\overline{k}_0r_0}\,:\, x_3=\Phi_0(x_1,x_2)\right\},
\end{equation*}
\begin{equation*}
    \partial D_1\cap B_{\overline{k}_0r_0}=\left\{
    (x_1,x_2,x_3)\in  B_{\overline{k}_0r_0}\,:\, x_3=\Phi_1(x_1,x_2)\right\},
\end{equation*}
where $\Phi_0$ and $\Phi_1$ are Lipschitz functions with Lipschitz constant bounded by $L_1$ and such that $\Phi_0(0,0)=0$ and $\Phi_1(0,0)=a_1$.

Furthermore, on $B_{k_1r_0}^\prime=\left\{
    (x_1,x_2)\in  \RR^2\,:\, x_1^2+x_2^2\leq k_1^2r_0^2\right\}$ we have
    \begin{equation*}
        \Phi_0(x_1,x_2)=l^0_1x_1+l^0_2x_2,\quad 
        \Phi_1(x_1,x_2)=l^1_1x_1+l^1_2x_2\quad \forall (x_1,x_2)\in B_{k_1r_0}^\prime
    \end{equation*}
    \begin{equation*}
        S_0=\left\{(x_1,x_2,\Phi_0(x_1,x_2)\,:\, (x_1,x_2)\in B_{k_1r_0}^\prime \right\}\subset \partial D_0,
    \end{equation*}
    \begin{equation*}
        S_1=\left\{(x_1,x_2,\Phi_1(x_1,x_2)\,:\, (x_1,x_2)\in B_{k_1r_0}^\prime \right\}\subset \partial D_1,
    \end{equation*}
    \begin{equation*}
        (l^0_1-l^1_1)^2+ (l^0_2-l^1_2)^2 \leq \left(\frac{Kd_H(D_0,D_1)}{r_0}\right)^2,
    \end{equation*}
    \begin{equation*}
        |a_1|\leq Kd_H(D_0,D_1),
    \end{equation*}
    and
    \begin{equation}\label{distPhi}
        \left|\Phi_0(x_1,x_2)-\Phi_1(x_1,x_2)\right|\geq K_1(d_H(D_0,D_1))^3, \quad \forall (x_1,x_2)\in B_{k_1r_0}^\prime.
    \end{equation}
\end{proposition}
\begin{remark}\label{rem5}
Notice that $k_0$ can be chosen such that $D_0$ and $D_1$ are on the same side with respect to $S_0$ and $S_1$.

We call $D_0$ the polyhedron for which the point $P_0\in \partial \Omega_{\mathcal{G}}$.

\end{remark}

Let us now introduce the description of a tubular neighborhood of a curve as was introduced in \cite{AS13,ADCMR}.
Let $P\in\partial\Omega_{\mathcal{G}}$ and let $\nu$ be a unit direction such that the line segment $[P,P+d\nu]$ is contained in $\mathcal{G}$ for some $d>0$. Let $\overline{P}$ be a point on $\partial\om$, consider a curve $\mathfrak{c}$ joining $\overline{P}$ to $P+d\nu$ and define, for some $R\in (0,d)$
\begin{equation*}
    V_{R}(\mathfrak{c})=\bigcup_{Q\in\mathfrak{c}}B_R(Q)\bigcup \mathcal{C}\left(P,\nu, \frac{d^2-R^2}{d},\arcsin\frac{R}{d}\right)
\end{equation*}
where $\mathcal{C}$ is the cone defined in \eqref{cone}.

In the next proposition, we show that such a set $V_{R}(\mathfrak{c})$ can be constructed in $\mathcal{G}$, see, for example, Figure \ref{fig:geometric_lemma}.

\begin{figure}[!h]
     \centering
     \includegraphics[width=10cm]{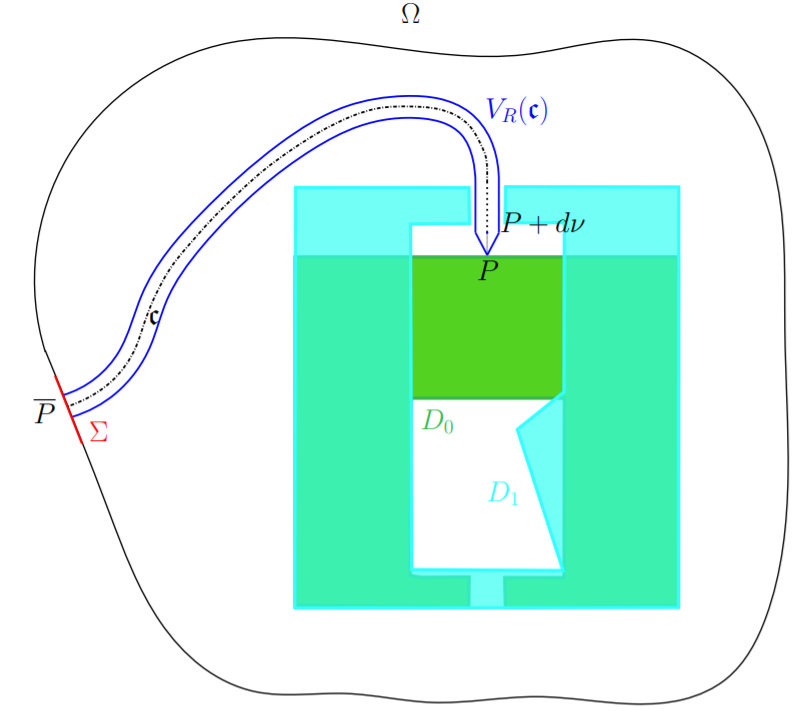}        \caption{A 2D-section of a possible configuration with the representation of $V_R(\mathfrak{c})$ and $\mathfrak{c}$ connecting $\overline{P}$ to $P+d\nu$.}\label{fig:geometric_lemma}
\end{figure}
\begin{proposition}\label{lemmageometrico}
If $D_0$, $D_1\in\mathfrak{D}$, there exist constants $C_3$, $d$, $R$ (with $R<d$) and $R_1$ depending only on the a priori data and there is a point $P\in\partial D_0\cap \partial\Omega_{\mathcal{G}}$ such that
\begin{equation}\label{1-19}
    C_3 d_\mu^3(D_0,D_1)\leq dist(P,D_1),
\end{equation}
\begin{equation}\label{infaccia}
  dist(P, \{\sigma^{D_0}_{ij} \}_{i\neq j}) \geq R_1,
\end{equation}
and such that, given any point $\overline{P}\in\partial\om$ there is a curve $\mathfrak{c}$ joining $\overline{P}$ to $P+d\nu$, where $\nu$ is the unit outer normal to $\partial D_0$,  such that 
\begin{equation}\label{2-19}
    V_{R}(\mathfrak{c})\subset \overline{\mathcal{G}}.
\end{equation}
\end{proposition}

\begin{proof}
Let us denote by $d_\mu=d_\mu(D_0,D_1)$ and let
\begin{equation*}
    d_1=\frac{r_0}{2C_1C_2}\min\left\{k_1,a\right\}.
\end{equation*}

We distinguish two cases. 

{\it Case 1}: $d_\mu\geq d_1$.
\newline
Note that, by Lemma 5.5 in \cite{ARRV}, the sets $E_t^{D_0}$ and $E_t^{D_1}$ (defined as in \eqref{EDt}) are connected for $t\leq d_1\leq ar_0$.

Let $\widetilde{P}$ be a point on $\partial D_0\cap \partial \Omega_{\mathcal{G}}$ be the point that satisfies
\begin{equation*}
 d_\mu(D_0,D_1)=dist(\widetilde{P},D_1).   
\end{equation*}
 Let $P$ be a point in the face containing $\widetilde{P}$ (or in one of the faces containing $\widetilde{P}$) such that \begin{equation*}
     dist(P, \{ \sigma_{ij}^{D_0}\}_{i\neq j})\geq \frac{d_1}{4}\sin \left(\frac{\theta_0}{2}\right)
 \end{equation*} 
 and 
 \begin{equation}\label{111}
     dist(P,D_1)\geq \frac{d_1}{2}.
 \end{equation}
 Consider the outer cone to $D_0$ at $P$ (see Remark \ref{remcono}). Since $P$ is internal to a face, the  direction $\nu$ can be chosen orthogonal to $\partial D_0$.

The point $P+\frac{d_1}{4}\nu$ belongs to $E_{D_1}^{\frac{d_1}{2}}$ and also to $E_{D_0}^{\frac{d_1}{4}\sin\vartheta}$ (with $\vartheta$ from Remark \ref{remcono}).

Then, given any point $\overline{P}\in \partial\om$ there is a curve $\mathfrak{c}$ joining $P+\frac{d_1}{4}\nu$ to $\overline{P}$ with distance bigger than $\frac{d_1}{4}\sin\vartheta$ from $\Omega_{\mathcal{G}}$.

Moreover we trivially have
\begin{equation*}
    d_\mu^3(D_0,D_1)\leq (diam(\om))^3\leq R_0^3,
\end{equation*}
hence, by \eqref{111}
\begin{equation*}
   d_\mu^3(D_0,D_1)\leq R_0^3 \frac{2dist(P,D_1)}{d_1} 
\end{equation*}
that gives \eqref{1-19}.

{\it Case 2: $d_\mu<d_1$}
\newline
 Since by Proposition \ref{prop1} and Proposition \ref{prop2} 
 \begin{equation*}
     d_H(D_0,D_1)\leq C_1 C_2 d_\mu(D_0,D_1)
 \end{equation*}
 we have
  \begin{equation*}
     d_H(D_0,D_1)\leq C_1 C_2 d_1 \leq k_0 r_0
 \end{equation*}
 so that the assumptions of Proposition \ref{prop4} hold true.
 
 
 Let now $P_0$ be the point in Proposition \ref{prop4} and let $\nu_0$ be the normal direction to $S_0$ (defined in Proposition \ref{prop4}). 
 Notice that, due to \eqref{distPhi} the cone of $\mathcal{C}(P_0,\nu_0, k_1r_0, \pi/2)$ is contained in $\mathcal{G}$.\\
 Let us take the point $P_0+\frac{k_1r_0}{2}\nu_0$ and notice that
 \begin{equation*}
     dist\left(P_0+\frac{k_1r_0}{2}\nu_0, D_0\right)=\frac{k_1r_0}{2}.
 \end{equation*}
 Let $t_0=\min\{\frac{k_1r_0}{2},ar_0\}$ so that
 $E_{t_0}^{D_0}$ is connected.
 
 Let $\mathfrak{c}$ be a curve joining $P_0+\frac{k_1r_0}{2}\nu_0$ to a point $\overline{P}\in\partial\om$ and such that $\mathfrak{c}\subset E_{t_0}^{D_0}$.
 By choosing $d=\frac{k_1r_0}{2}$ and $R=t_0/4$ the tubular set $V_{R}(\mathfrak{c})$ (starting from $P=P_0$) is contained in $\overline{\om\setminus D_0}$.

Now, since 
\begin{equation*}
    d_H(D_0,D_1)\leq C_1C_2d_1\leq \frac{t_0}{2}
\end{equation*}
the set $V_{R}(\mathfrak{c})$ is contained also in $\overline{\om\setminus D_1}$, and \eqref{2-19} follows.

Inequalities \eqref{1-19} and \eqref{infaccia} (for $P=P_0$ and $R_1=k_0r_0$) are a straightforward consequence of \eqref{distPhi}, \eqref{14} and \eqref{12}.
\end{proof}
\subsection{Estimating the distance between vertices of close polyhedra}
We now state and prove the main result of the section: if two polyhedra in $\mathfrak{D}$ are close enough, then they have the same number of vertices (and faces and edges). 

\begin{proposition}\label{distvert}
There exist two positive constants $\delta_0$ and $C$ depending only on the a priori data, such that, if for some $D_0$ and $D_1$ in $\mathfrak{D}$,
\begin{equation*}
    d_H\left(\partial D_0, \partial D_1\right)\leq \delta_0,
\end{equation*}
then $D_0$ and $D_1$ have the same number $N$ of vertices $\left\{ V^{D_0}_i\right\}_{i=1}^N$ and $\left\{V^{D_1}_i\right\}_{i=1}^N$, respectively, which can be ordered in such a way that
\begin{equation}\label{eq:distvert}
    dist\left(V^{D_0}_i,V^{D_1}_i\right)\leq C d_H\left(\partial D_0, \partial D_1\right).
\end{equation}
Moreover, for each edge or face in $D_0$ there is an edge or a face in $D_1$ with corresponding vertices.
\end{proposition}
\begin{proof}
The proof of Proposition \ref{distvert} follows the same idea of the proof of Proposition 3.3 in \cite{BerFra20} in the two dimensional setting. In that case we show that, if the Hausdorff distance between the boundaries is small enough, a vertex of one of the two polygons cannot be too far from vertices of the other polygon without violating the a priori assumptions.

For polyhedra the proof is more involved and it is divided in two steps: in the first step we show that the distance between an arbitrary vertex in $D_0$ from the edges of $D_1$ can be bounded by $ C d_H\left(\partial D_0, \partial D_1\right)$ where $C$ depends only on the a priori data. The main idea to prove this consists in showing that a small neighborhood of a face of one polyhedron cannot contain a vertex of the second polyhedron since the length of edges and width of angles are bounded from below by the a priori data. 

In the second step, we show that an arbitrary vertex of $D_0$ has distance smaller than $ C d_H\left(\partial D_0, \partial D_1\right)$ from a vertex in $D_1$. This time the idea is that a small neighborhood of a pair of intersecting faces cannot contain a vertex that does not violate assumption \eqref{angolinterni}. 
Since assumption \eqref{lunghlati} holds, if $d_H\left(\partial D_0, \partial D_1\right)$ is small enough there is a one to one correspondence between vertices of the two polyhedra.

For sake of brevity let us denote by
\begin{equation}\label{dH}
    d_H=d_H\left(\partial D_0, \partial D_1\right),
\end{equation}
and let
\begin{equation*}
    \left(\partial D_1\right)^{(d_H)}=\left\{x\in\RR^3\,:\,dist(x,\partial D_1)\leq d_H\right\}.
\end{equation*}
By definition of Hausdorff distance it follows that $\partial D_0\subset  \left(\partial D_1\right)^{(d_H)}$.

We can also assume that $\left(\Omega\right)_{d_H}\setminus \left(\partial D_1\right)^{(d_H)}$ is connected by \cite[Lemma 5.5]{ADCMR}, where
\begin{equation*}
    \left(\Omega\right)_{d_H}=\left\{x\in\Omega\,:\, dist(x,\partial\Omega)>d_H\right\}.
\end{equation*}

Let us choose an arbitrary vertex in $D_0$ and let us denote it by $V_1^{D_0}$. Let 
$F_i^{D_1}$ be a face of $D_1$ such that
\begin{equation*}
    dist(V_1^{D_0},F_i^{D_1})\leq d_H
\end{equation*}
(notice that such a face exists because $V_1^{D_0}\in  \left(\partial D_1\right)^{(d_H)}$).

Let us choose our coordinate system such that $V_1^{D_0}=(0,0,0)$, and $F_i^{D_1}$ lies on the plane $\{x_3=-c\}$ for $0\leq c\leq d_H$.

We now want to show that there exists a vertex (say $V_1^{D_1}$) of the polygon $F_i^{D_1}$ such that 
\begin{equation*}
    dist(V_1^{D_0},V_1^{D_1})\leq Cd_H 
\end{equation*}
where $C$ depends only on the a priori assumptions.

\textit{First step.}  Let us show that there exists $C_0$, depending only on the a priori data, such that, if $d_H$ is small enough, then
\begin{equation}\label{1-5}
    dist((0,0,-c),\partial F_i^{D_1})\leq C_0d_H
\end{equation}
and, hence, since $0\leq c\leq d_H$
\begin{equation}\label{3-5}
    dist(V_1^{D_0},\partial F_i^{D_1})\leq (C_0+1)d_H.
\end{equation}
In order to prove \eqref{1-5}, let us assume that
\begin{equation}\label{2-5}
    dist((0,0,-c),\partial F_i^{D_1})> C_0d_H
\end{equation}
and show that there is a constant $C_0$ such that \eqref{2-5} leads to a contradiction for sufficiently small $d_H$.

By assumption \eqref{angolinterni}, the cones with basis $B^\prime_{C_0d_H}\left((0,0,-c)\right)$ and height $C_0d_H\tan \theta_0$ do not intersect other faces of $D_1$ except $F_i^{D_1}$.

Let us take $C_0>\frac{1+\cos\theta_0}{\sin \theta_0}$. It is easy to show that the ball centered at $V_1^{D_0}$ with radius $C_1d_H$, where
$C_1=\frac{1}{2}\left(C_0\sin\theta_0-\cos\theta_0-1\right)$ does not intersect the set
\begin{equation*}
    \left\{x\in\RR^3\, :\, dist\left(x,\partial D_1\setminus F_i^{D_1}\right)\leq d_H\right\}.
\end{equation*}
Let us now take $d_H$ such that $C_1d_H<r_0$. This implies that the edges of $D_0$ that contain $V_1^{D_0}$, that are contained in $\left(\partial D_1\right)^{(d_H)}$, by definition of the Hausdorff measure, intersect $\partial B_{C_1d_H}\left(V_1^{D_0}\right)$ at points that lie between the planes $\pi^+=\left\{x_3=2d_H\right\}$ and $\pi^-=\left\{x_3=-2d_H\right\}$ (as a matter of fact the region on the ball that can contain these intersections is smaller, but we choose  this one to have a symmetric one).
\begin{figure}[!h]
     \centering
     \includegraphics[width=7cm]{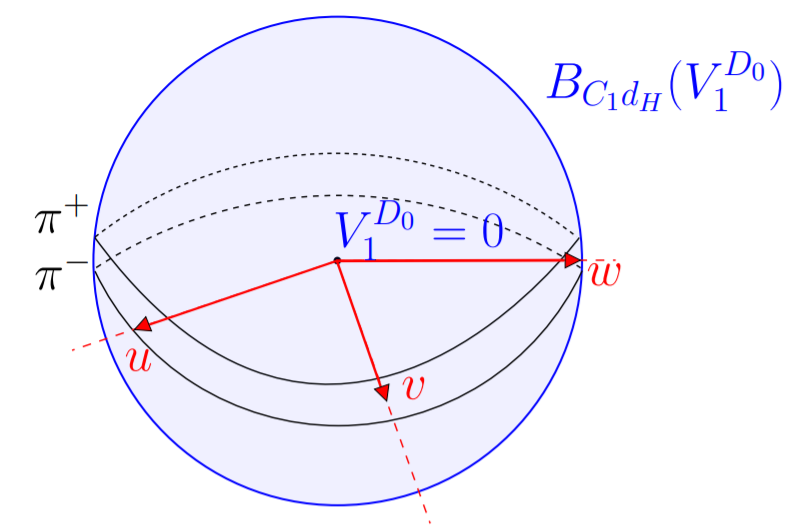}        
     \caption{First step of the proof. Note that if the strip is too small then the vectors are not all contained in it.}\label{fig:sphere}
\end{figure}

Let $\sigma_{ij}^{D_0}$ be one of the edges of $D_0$ that contains $V_1^{D_0}$ and let us denote by $v$ the position vector that represents the intersection of this edge with the sphere $B_{C_1d_H}\left(V_1^{D_0}\right)$.

Let $u$ and $w$ the position vectors with tips at the intersection of the edges of the  faces of $D_0$ adjacent to $\sigma_{ij}^{D_0}$ with $\partial B_{C_1d_H}\left(V_1^{D_0}\right)$

We have that
\begin{equation}\label{1-8}
\begin{split}
    |u|=|v|=|w|=C_1d_H,\\
    |u_3|,\, |v_3|,\, |w_3|\leq 2d_H.
\end{split}
\end{equation}
Let $\theta_{ij}$ denote the internal angle at the edge $\sigma_{ij}^{D_0}$.  We now show that, if $C_1$ (and, hence, if $C_0$) is big enough and $C_1d_H<r_0$, then 
\begin{equation*}
    |\cos \theta_{ij}|>\cos \theta_0
\end{equation*}
in contradiction with assumption \eqref{angolinterni}.

Let us consider the unit normal direction to the faces intersecting at $\sigma_{ij}^{D_0}$:
\begin{equation}
    t=\frac{u\times v}{|u\times v|}\mbox{ and }
    \tau=\frac{v\times w}{|v\times w|}.
\end{equation}
Notice that, by \eqref{1-8},
\begin{eqnarray*}
    |u\times v|^2&\leq& (u_1^2+u_2^2+v_1^2+v_2^2)(u_3^2+v_3^2)+ (u_1v_2-u_2v_1)^2\\
    &\leq& 16C_1^2d_H^4 +t_3^2|u\times v|^2,
\end{eqnarray*}
so that 
\begin{equation}\label{1-10}
    (1-t_3^2)|u\times v|^2\leq 16C_1^2d_H^4.
\end{equation}
From assumption \eqref{angolifacce} and \eqref{1-8} we have
\begin{equation}\label{2-10}
    |u\times v|^2=|u|^2|v|^2|\sin \hat{uv}|^2\geq (C_1d_H)^4 \sin^2\theta_0.
\end{equation}
From \eqref{1-10} and \eqref{2-10},
\begin{equation*}
    1-t_3^2\leq \frac{16}{C_1^2\sin^2\theta_0},
\end{equation*}
hence
\begin{equation}\label{3-10}
    1-|t_3|\leq \frac{16}{C_1^2\sin^2\theta_0}
\end{equation}
and, in the same way,
\begin{equation}\label{4-10}
    1-|\tau_3|\leq \frac{16}{C_1^2\sin^2\theta_0}.
\end{equation}
Now, by \eqref{3-10} and \eqref{4-10},
\begin{eqnarray}
    |\cos \theta_{ij}|&=&|t\cdot\tau|\geq |t_3||\tau_3|-|t^\prime||\tau^\prime|\\
&=&|t_3||\tau_3|-\sqrt{(1-|t_3|^2)(1-|\tau_3|^2)}\geq 1-\frac{48}{C_1^2\sin^2\theta_0}.\end{eqnarray}
For this reason, if 
\begin{equation}\label{1-11}
    C_1^2>\frac{48}{(1-\cos\theta_0)\sin^2\theta_0},
\end{equation}
we have that
\begin{equation*}
    |\cos\theta_{ij}|>\cos\theta_0
\end{equation*}
that contradicts \eqref{angolinterni}.

So, let us take, for example
\begin{equation*}
    C_0=2\left\{\frac{8\sqrt{3}}{\sqrt{1-\cos\theta_0}\sin^2\theta_0}+\frac{1+\cos\theta_0}{\sin\theta_0}\right\}.
\end{equation*}
With this choice, \eqref{1-11} holds, hence we have a contradiction for $d_H<\frac{r_0}{C_1}$. This implies that, for $d_H<\frac{r_0}{C_1}$,  \eqref{1-5} and \eqref{3-5} hold.

\textit{Second step.}  Since  \eqref{3-5} holds, there is an edge $\sigma_{ij}^{D_1}$ such that
\begin{equation*}
    dist\left(V_1^{D_0},\sigma_{ij}^{D_1}\right)\leq (C_0+1)d_H.
\end{equation*}
Let $V_1^{D_1}$ and $V_2^{D_1}$ be the endpoints of 
$\sigma_{ij}^{D_1}$.

We want to show that there is $C_2$ depending only on the a  priori data, such that, for $d_H$ small enough, either
\begin{equation*}
    dist\left(V_1^{D_0},V_1^{D_1}\right)\leq C_2d_H\mbox{ or }dist\left(V_1^{D_0},V_2^{D_1}\right)\leq C_2d_H.
\end{equation*}
Again, we proceed by contradiction and assume that
\begin{equation*}
    dist\left(V_1^{D_0},V_1^{D_1}\right)> C_2d_H\mbox{ and }dist\left(V_1^{D_0},V_2^{D_1}\right)> C_2d_H.
\end{equation*}
and get a contradiction with the a priori assumptions on $\mathfrak{D}$, see Figure \ref{fig:setting_step2}.
\begin{figure}[!h]
     \centering
     \begin{subfigure}[b]{0.3\textwidth}
         \centering
         \includegraphics[width=\textwidth]{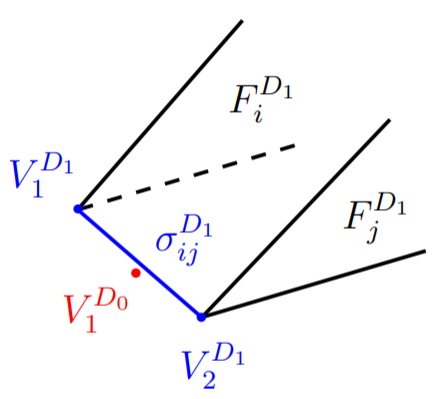}
         \caption{Geometrical setting.}
         \label{fig:polyhedronD0}
     \end{subfigure}
     \hfill
     \begin{subfigure}[b]{0.3\textwidth}
         \centering
         \includegraphics[width=\textwidth]{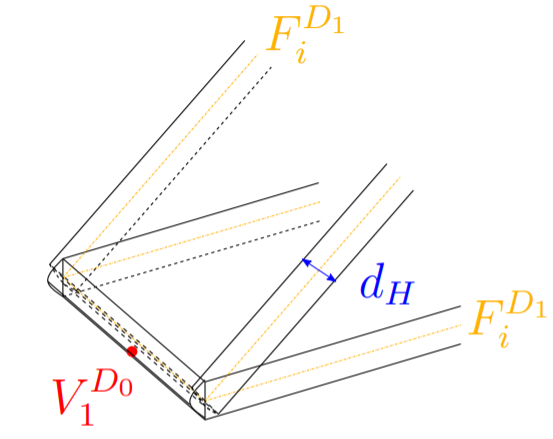}
         \caption{A partial $d_H$-neighborhood of the two faces $F^{D_1}_i$ and $F^{D_1}_j$.}
         \label{fig:polyhedronD1}
     \end{subfigure}
     \hfill
     \begin{subfigure}[b]{0.3\textwidth}
         \centering
         \includegraphics[width=\textwidth]{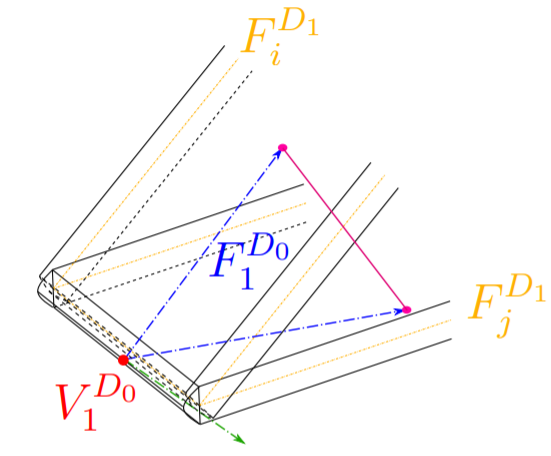}
         \caption{The two edges in blue of the face $F^{D_0}_1$.}
         \label{fig:configuration_second_step}
     \end{subfigure}
    \caption{A sketch of the geometrical setting for the second step of the proof.}
        \label{fig:setting_step2}
\end{figure}

Let $F_j^{D_1}$ be such that $F_i^{D_1}\cap F_j^{D_1}=\sigma_{ij}^{D_1}$.

As in the first step, by elementary calculations, there is a ball centered at $V_1^{D_0}$ of radius $C_3d_H$,  where $C_3$ depends on the a priori data and on $C_2$, such that
\begin{equation*}
    B_{C_3d_H}\left(V_1^{D_0}\right)\cap\left\{x\in\RR^3\, :\, dist\left(x,\partial D_1\setminus\left(F_i^{D_1}\cup F_j^{D_1}\right)\right)\leq d_H\right\}=\emptyset.
\end{equation*}
This implies that the intersections of all the edges containing $V_1^{D_0}$ with such ball lie in a $d_H$-neighborhood of the two faces. 
From the first step, we know that, it is not possible to have all these intersections in the neighborhood of only one of the two faces.
Hence, there is a face (say $F_1^{D_0}$) containing $V_1^{D_0}$ that has one edge in the neighborhood of $F_i^{D_1}$ and one in $F_j^{D_1}$. With calculations similar to the ones in the first step, that we omit for sake of shortness, it is possible to show that, for $C_2$ big enough, a part of the face $F_1^{D_0}$ does not belong to $\left(\partial D_1\right)^{(d_H)}$ contradicting the definition of the Hausdorff distance.
\end{proof}
\section{A first rough stability estimate}\label{sec4}
In this section we derive a rough stability estimate of polyhedral inclusions measured in the Hausdorff distance in terms of the operator norm of the partial DtN map. As shown in the previous section, this estimate is crucial to prove that the two polyhedra have the same vertices which can be ordered in such a way that they are close, see Proposition \ref{distvert}. 
\subsection{On some properties of the Green's function}
Let us first recall Alessandrini's identity. Let $u_0$ and $u_1$, with $\textrm{supp}(u_0\lfloor_{\partial\Omega})$,  $\textrm{supp}(u_1\lfloor_{\partial\Omega}) \subset\Sigma$, be solutions of the equations
\begin{equation*}
    \dive(\gamma_{D_i}\nabla u_i)=0,\qquad i=0,1,
\end{equation*}
and $\Lambda_{\gamma_{D_i}}^{\Sigma}$ the corresponding local DtN maps, for $i=0,1$. Then, it holds
\begin{equation}\label{eq:Alessandrini_identity}
    \big\langle (\Lambda_{\gamma_{D_0}}^{\Sigma}-\Lambda_{\gamma_{D_1}}^{\Sigma})u_0\big\lfloor_{\Sigma}, u_1\big\lfloor_{\Sigma} \big\rangle=\int_{\om}(k-1)(\chi_{D_0}-\chi_{D_1})\nabla u_0 \cdot \nabla u_1\, dx,
\end{equation}
where $\chi_{D_i}$, for $i=0,1$, is the characteristic function of $D_i$.

As in \cite{AleKim12,ARRV}, we introduce an augmented domain $\oms$, attaching to $\om$ an open set $\om_0$, in its exterior, whose boundary intersects $\partial\om$ on an open portion $\Sigma_0\Subset\Sigma$ such that $\Sigma_0$ has size which is a fraction of $r_0$. Let us choose $\Omega_0$ in such a way that $\oms:=\om\cup\Sigma_0\cup\om_0$ has the following properties: there exist $r_1$, $M_1$, depending only on $r_0$, and $M_0$
such that
\begin{enumerate}
  \item $\oms$ is open, connected with Lipschitz boundary with constants $r_1$, $M_1$;
  \item there exists $P_0\in\om_0$ such that
    \begin{equation}\label{eq:ball}
        B_{2r_1}(P_0)\subset\om_0.
    \end{equation}
\end{enumerate}
We extend the conductivity to be $1$ in $\om_0$ still denoting it with  $\gamma_D$.\\
Let $\Gamma(x,y)$ be the fundamental solution of the Laplace operator, that is the function
\begin{equation*}
    \Gamma(x,y)=\frac{1}{4\pi}\frac{1}{|x-y|},
\end{equation*}
and with $\Gs$ the Green's function, solution to 
\begin{equation}\label{eq:equ_green_func}
\begin{cases}
    \dive(\gamma_D\nabla \Gs(\cdot,y))=-\delta(\cdot-y) & \textrm{in}\  \oms \\
    \Gs(\cdot,y)=0 & \textrm{on}\ \partial\oms,
\end{cases}
\end{equation}
where $\delta(\cdot-y)$ is the Dirac distribution centered in $y$. Let us recall some properties of the Green function. For all $x,y\in\oms,\ x\neq y$, it holds 
\begin{equation*}
\begin{aligned}
        \Gs(x,y)=&\Gs(y,x)\\
        0<\Gs(x,y)\leq& \frac{c}{|x-y|},
\end{aligned}
\end{equation*}
where $c$ depends only on $k$, see \cite{AleVes05}. Fix a point $y\in\oms\setminus D$ and let $0<r_2=dist(y,\{\sigma_{ij}^{D}\}_{i\neq j}\cup \partial\oms)$, for $i,j=1,\cdots, H$. Then, there exists a constant $C>1$ depending only on the a priori data such that $B_{r_2/C}(y)$ contains at most a portion of one face of the polyhedron $D$. Hence, in this case, the ball is divided into two zones with different conductivity coefficient (thanks to \eqref{eq:conductivity_coeff}), that is, for a suitable coordinate system, there exists $a\in [-\frac{r_2}{C},\frac{r_2}{C}]$ such that
\begin{equation}\label{eq:gammahat}
    \widehat{\gamma}_D(x)=1+(k-1)\chi_{\{x_3>a\}}(x),\qquad \forall x\in B_{r_2/C}(y).
\end{equation}
We extend the coefficient $\widehat{\gamma}_D$ in $\mathbb{R}^3$, that is, we define
\begin{equation*}
    \widehat{\gamma}_y(x):=1+(k-1)\chi_{\{x_3>a\}}(x),\qquad \forall x\in \mathbb{R}^3,
\end{equation*}
where the same coordinate frame of \eqref{eq:gammahat} has been used.
Denote by $\widehat{\Gamma}$ the biphase fundamental solution of 
\begin{equation*}
    \dive(\widehat{\gamma}_y(\cdot)\nabla\widehat{\Gamma}(\cdot,y))=-\delta(\cdot-y),\qquad \textrm{in}\ \mathbb{R}^3.
\end{equation*}
We refer the reader to \cite{AleVes05} for more details on the biphase fundamental solution. In the following proposition,  we recall other useful properties of the Green function 
that come from some of the results in \cite{AleVes05,BerFra11,BerFraVes21}. 

\begin{proposition}
For all $C_1>1$ there exists a constant $C>0$ depending on the a priori data and $C_1$ such that, for all $y\in\oms\setminus D$ satisfying
\begin{equation*}
    dist(y,\{\sigma_{ij}^{D}\}_{i\neq j}\cup \partial\oms)\geq \frac{r_0}{C_1},\qquad i,j=1,\cdots,H,
\end{equation*}
it follows that  
\begin{equation}\label{eq:estimate_green_fundamental_solution}
    \|\Gs(\cdot,y)-\widehat{\Gamma}(\cdot,y)\|_{H^1(\oms)}\leq C,
\end{equation}
and, for all $\varrho>0$,
\begin{equation}\label{eq:H1_estim_green_func}
    \|\Gs(\cdot,y)\|_{H^1(\oms\setminus B_{\varrho}(y))}\leq C \varrho^{-\frac{1}{2}}.
\end{equation}
Let $P\in\partial D$. Without loss of generality, assume that $P$ belongs to the $i-th$ face $F_i$ and that 
\begin{equation*}
    dist(P,\{\sigma_{ij}^{D}\}_{i\neq j})\geq R_1,
\end{equation*}
and let $y_r=P+r\nu(P)$, $r>0$, where $\nu(P)$ is outer unit normal vector in $P$ to $\partial D$. Then, for all $r<\frac{R_1}{2}$, and $x\in D\cap B_{\frac{R_1}{2}}(P)$, we get that
\begin{equation}\label{eq:unif_est_green_fund}
    |\nabla \Gs(x,y_r)-\nabla\widehat{\Gamma}(x,y_r)|\leq C,
\end{equation}
where $\nabla \widehat{\Gamma}=\frac{2}{k+1}\nabla\Gamma(x,y_r)$.
\end{proposition}
\subsection{Estimating an auxiliary function}
Recalling \eqref{eq:omD}, for all $y,z\in\mathcal{G}$, we consider 
\begin{equation}\label{eq:func_S}
    S(y,z):=(k-1)\int_{\om}(\chi_{D_0}-\chi_{D_1})\nabla \Gs_0(x,y)\cdot \nabla \Gs_1(x,z)\, dx,
\end{equation}
where $\Gs_i$, for $i=0,1$, are solutions to \eqref{eq:equ_green_func}, where $\gamma_D=\gamma_{D_i}$. Note that
\begin{itemize}
    \item for all $z\in\oms\setminus\Omega_{\mathcal{G}}$,
    \begin{equation}
        \Delta_y S(\cdot,z)=0,\qquad \textrm{in}\ \oms\setminus\Omega_{\mathcal{G}};
    \end{equation}
    \item for all $y\in\oms\setminus\Omega_{\mathcal{G}}$,
    \begin{equation}
        \Delta_z S(y,\cdot)=0,\qquad \textrm{in}\ \oms\setminus\Omega_{\mathcal{G}};
    \end{equation}
    \item for all $y,z\in\Omega_0$, the Green's functions $\Gs_0$ and $\Gs_1$ do not have singularities in $\om$ and by the regularity of $\Gs_0$, $\Gs_1$ in $\oms\setminus\Omega_{\mathcal{G}}$,
    \begin{equation*}
        \Gs_0(\cdot,y)\lfloor_{\partial\om},\ \Gs_1(\cdot,z)\lfloor_{\partial\om}\in H^{\frac{1}{2}}_{co}(\Sigma), 
    \end{equation*}  
    that is, thanks to \eqref{eq:H1_estim_green_func},
    \begin{equation}\label{eq:bound_estimates_green_func}
        \|\Gs_0(\cdot,y)\|_{H^{\frac{1}{2}}_{co}(\Sigma)}, \|\Gs_1(\cdot,z)\|_{H^{\frac{1}{2}}_{co}(\Sigma)}\leq C,
    \end{equation}
    where $C$ depends only on the a priori data. In fact, for example
    \begin{equation}\label{eq: estimate_G_boundary}
        \begin{aligned}
            \|\Gs_0(\cdot,y)\|_{H^{\frac{1}{2}}_{co}(\Sigma)}&\leq \|\Gs_0(\cdot,y)\|_{H^{\frac{1}{2}}(\partial\om)}\\
            &\leq \|\Gs_0(\cdot,y)\|_{H^1(\om)}\leq \|\Gs_0(\cdot,y)\|_{H^1(\oms\setminus B_{r_1}(y))}\leq C r_1^{-\frac{1}{2}}. 
        \end{aligned}
    \end{equation}
    Analogously for $\Gs_1(\cdot,z)$.
\end{itemize}
In order to prove stability estimates in terms of the Hausdorff distance of the inverse problem under investigation, we need first to establish upper and lower bounds for the function $S(y,z)$ defined in \eqref{eq:func_S}. These are contained in the next two propositions. 
To simplify the presentation, we assume, without loss of generality, that using a rigid transformation of coordinates the point $P$ in Proposition \ref{lemmageometrico} coincides with the origin, i.e. $P=O$, and the outer unit normal vector $\nu$ is equal to $e_3$, where $e_3=(0,0,1)$. Moreover, in accordance to Definition \ref{def1}, we use the notation $\sigma^{\Omega_{\mathcal{G}}}_{ij}$, with $i\neq j$, to denote the edges of $\Omega_{\mathcal{G}}$.\\
The proofs of the following two propositions are in Appendix \ref{appendix:stability}. 
\begin{proposition}\label{th: stability_above}
Assume that 
\begin{equation}\label{eq:smallness_diff_DtN_map}
\left\| \Lambda^{\Sigma}_{\gamma_{D_0}}-\Lambda^{\Sigma}_{\gamma_{D_1}}\right\|_{\star}\leq \varepsilon,
\end{equation}
where $0<\varepsilon<1$. Under the notation of Proposition \ref{lemmageometrico}, let $Q=P+de_3$ and $\xi_h=P+he_3$, where $0<h<d_1$ and 
\begin{equation}\label{eq:d1_and_vartheta}
d_1:=d\left(1-\frac{\sin\vartheta}{4}\right)\qquad \textrm{and}\qquad \vartheta=\arctan\left(\frac{R}{d}\right).   
\end{equation}
Then, there exists two suitable constants $C_3$ and $C_4$ depending on the a priori data such that
\begin{equation}\label{eq:estimate_S}
    \big|S(\xi_h,\xi_h)\big|\leq \frac{C}{h}\varepsilon^{C_3 h^{C_4}}
\end{equation}
where $C$ depends on the a priori data. 
\end{proposition}
\begin{proposition}\label{th: stability_below}
Under the notation of Proposition \ref{lemmageometrico}, let $\xi_h=P+h e_3$. There exist $0<\overline{h}<\frac{1}{2}$ and $0<\overline{C}<1$ depending only on the a priori data such that
\begin{equation}\label{eq:estimate_from_below}
    |S(\xi_h,\xi_h)|\geq \frac{C}{h}\,\qquad \forall h,\ 0<h\leq \overline{h}\varrho,
\end{equation}
where 
\begin{equation}\label{eq:condition_on_varrho}
\varrho=\min\{dist(P,D_1),\overline{C}r_0\}    
\end{equation}
and $C$ depends on the a priori data. 
\end{proposition}
\begin{remark}
Note that $\varrho\leq \overline{C}r_0$ is needed in order to guarantee that a ball of center $P$ and radius $\varrho$ doesn't intersect edges and vertices of $\partial \Omega_{\mathcal{G}}\setminus \partial\Omega$.
\end{remark}
\subsection{Logarithmic stability estimates}
Now, we use Proposition \ref{th: stability_above} and Proposition \ref{th: stability_below} to prove the following logarithmic stability estimate.
\begin{theorem}\label{stablog}
Let the assumptions of Section \ref{sec:assumptions} apply. Let $D_0$, $D_1$ be two polyhedral inclusions in $\mathfrak{D}$. Let $1$ and $k$ be the conductivity coefficients of $\om\setminus D_i$ and $D_i$, for $i=0,1$, respectively.
If, for some $\varepsilon$ with $0<\varepsilon<1$,
\begin{equation*}
   \left\|\Lambda^{\Sigma}_{\gamma_{{D}_0}}-\Lambda^{\Sigma}_{\gamma_{{D}_1}}\right\|_{\star}\leq \varepsilon,
\end{equation*}
then
\begin{equation}\label{eq:stability_estimate}
    d_H(\partial D_0,\partial D_1)\leq \widetilde{\omega}(\varepsilon),
\end{equation}
where $\widetilde{\omega}(\varepsilon)$ is an increasing function in $[0,+\infty)$ such that
\begin{equation*}
    \widetilde{\omega}(t)\leq C |\log t|^{-\zeta},\qquad \textrm{for all}\,\,  0<t<1,
\end{equation*}
where $C>0$ and $\zeta$, $0<\zeta\leq 1$ are constants depending only on the a priori data.
\end{theorem}
\begin{proof}
By \eqref{eq:estimate_S} and \eqref{eq:estimate_from_below}, we have
\begin{equation*}
    \frac{\widetilde{C}}{h}\leq |S(\xi_h,\xi_h)|\leq \frac{\widehat{C}}{h} \varepsilon^{C_3 h^{C_4}}, \qquad \forall h,\ 0<h\leq\overline{h}\varrho,
\end{equation*}
that is
\begin{equation*}
    C\leq \varepsilon^{C_3 h^{C_4}},
\end{equation*}
where $C_3$, $C_4$ are the constants in \eqref{eq:estimate_S} and $0<C<1$. Since $0<\varepsilon<1$, from the last inequality we get
\begin{equation*}
    h\leq \widetilde{C} \left(\frac{1}{|\log\varepsilon|}\right)^{\frac{1}{C_4}},\qquad \forall h,\ 0<h\leq \overline{h}\varrho. 
\end{equation*}
In particular, choosing $h=\overline{h}\varrho$, we find
\begin{equation*}
    \varrho \leq C\left(\frac{1}{|\log\varepsilon|}\right)^{\frac{1}{C_4}}.
\end{equation*}
From \eqref{eq:condition_on_varrho}, we have to distinguish two cases.\\
\textit{Case 1: $\varrho=dist(P,D_1)$}. In this case, by \eqref{1-19}, we get
\begin{equation*}
    d^3_{\mu}(D_0,D_1)\leq \varrho\leq C \left(\frac{1}{|\log\varepsilon|}\right)^{\frac{1}{C_4}},
\end{equation*}
that is
\begin{equation*}
    d_{\mu}(D_0,D_1)\leq \varrho^{\frac{1}{3}}\leq C \left(\frac{1}{|\log\varepsilon|}\right)^{\frac{1}{3C_4}}.
\end{equation*}
Therefore, thanks to Proposition \ref{prop2}, we find
\begin{equation*}
    d_H(\partial D_0,\partial D_1) \leq C d_{\mu}(D_0,D_1)\leq \varrho^{\frac{1}{3}}\leq C \left(\frac{1}{|\log\varepsilon|}\right)^{\frac{1}{3C_4}}.
\end{equation*}
\textit{Case 2: $\varrho=\overline{C}r_0$}. Then, we obtain the assertion of the theorem simply noticing that
\begin{equation*}
    d_H(\partial D_0,\partial D_1) \leq diam(\om)\leq C r_0 \leq\frac{C}{\overline{C}} \left(\frac{1}{|\log\varepsilon|}\right)^{\frac{1}{C_4}},
\end{equation*}
where $C$ depends on the a priori data only.
\end{proof}

\section{On the regularity properties of the local DtN map}\label{sec5}
In this section we investigate the differentiability properties of the local DtN map. The first part of this section is devoted to the non trivial task of constructing a Lipschitz vector field $\mathcal{U}$ from $\mathbb{R}^3$ to $\mathbb{R}^3$ mapping $D_0$ to $D_1$ which is piecewise affine in a neighborhood of  $\partial D_0$ (Proposition \ref{vectorfield}) and to prove its main properties, see Proposition \ref{prop:property_Phi}.  Then in Proposition \ref{shapederivative} and Proposition \ref{prop7.4} we state the differentiability of the DtN map showing that its Gateaux derivative along  the direction $\mathcal{U}$ exists and is continuous. Furthermore, we derive a distributed formula for the Gateaux derivative and we use this representation to bound it from below (Proposition \ref{prop:lower_bound_F'}).

\subsection{Construction of a Lipschitz vector field mapping $D_0$ to $D_1$}
In this subsection we assume that 
\begin{equation}\label{vertices}
d_H(\partial D_0, \partial D_1)\leq \delta_0
\end{equation}
as in Proposition \ref{distvert}, hence it follows that the two polyhedra $D_0$ and $D_1$ have the same number of vertices such that
$$
dist(V^{D_0}_i,V^{D_1}_i)\leq C d_H(\partial D_0, \partial D_1),\qquad \textrm{for}\ i=1,\dots,N.
$$
For sake of shortness we again use the notation \eqref{dH}
\begin{equation*}
    d_H=d_H(\partial D_0, \partial D_1)
\end{equation*}

Let $\mathcal{W}\subset \om$ be a tubular neighborhood of $\partial D_0$ with width $\frac{r_0}{4}$ so that 
\begin{equation*}
    dist(\mathcal{W},\partial\om)\geq \frac{r_0}{2}.
\end{equation*}
In the sequel, we denote by $\mathcal{T}_0$ the union of non overlapping isosceles triangles contained in the faces of $D_0$ with basis on the sides of the polyhedron and height 
\begin{equation}\label{eq:h0}
h_0=\frac{r_0\min\{1, \tan(\theta_0/2)\}}{2}.
\end{equation} 
The following result holds:
\begin{proposition}\label{vectorfield}
There exists a vector field $\mathcal{U}:\mathbb{R}^3\rightarrow \mathbb{R}^3$ with $\mathcal{U}\in W^{1,\infty}(\mathbb{R}^3)$  and satisfying the following properties 
\begin{align}
    &\mathcal{U}(V^{D_0}_i)=V^{D_1}_i-V^{D_0}_i,\qquad \forall\ i=1,\ldots,N,\label{ass1:U}\\
   &\textrm{supp}\ \mathcal{U}\subset \overline{\mathcal{W}}, \label{ass2:U}\\
   & \mathcal{U}\ \text{continuous,piecewise affine on }  \mathcal{T}_0 ,\label{ass3:U}\\
   & |\mathcal{U}|+| D\mathcal{U}| \leq \widetilde{C} d_H,\label{ass4:U}
\end{align}
where $D\mathcal{U}$ dentoes the Jacobian matrix of $\mathcal{U}$ and $\widetilde{C}$ is a constant depending only on the a priori constants.  
\end{proposition}
\begin{proof}
To construct the vector field $\mathcal{U}$ satisfying \eqref{ass1:U} - \eqref{ass4:U} observe that by Kirszbraun's theorem \cite[Theorem 1.31]{Sch69} it is always possible to extend a function $f:A\subset \mathbb{R}^3\to \mathbb{R}^3$ which is Lipschitz continuous on an arbitrary subset $A$ of $\mathbb{R}^3$ to a Lipschitz function $\bar{\mathcal{U}}:\mathbb{R}^3\to \mathbb{R}^3$ such that 
\begin{equation*}
    \bar{\mathcal{U}}\lfloor_A=f,
\end{equation*}
and $\bar{\mathcal{U}}$ having the same Lipschitz constant $L$ as $f$.\\ 
So, let us first construct the map $f$. We fix an arbitrary face $F_j^0$ of the polyhedron $D_0$. Assume that $F_j^0$ has $K$ sides. Then  on each side $l_i, i=1,\dots K$, we construct  isosceles triangles $\{T^0_i\}_{i=1}^K$, with basis  $l_i, i=1,\dots K$ and  height $h_0$, as defined in \eqref{eq:h0}, in such a way that all the triangles are strictly contained in $F_j^0$, disjoint and mutually intersecting only at the common vertex of $F_j^0$, see, for example, Figure \ref{fig:triangles}. 
\begin{figure}[!ht]
     \centering
     \includegraphics[width=5cm]{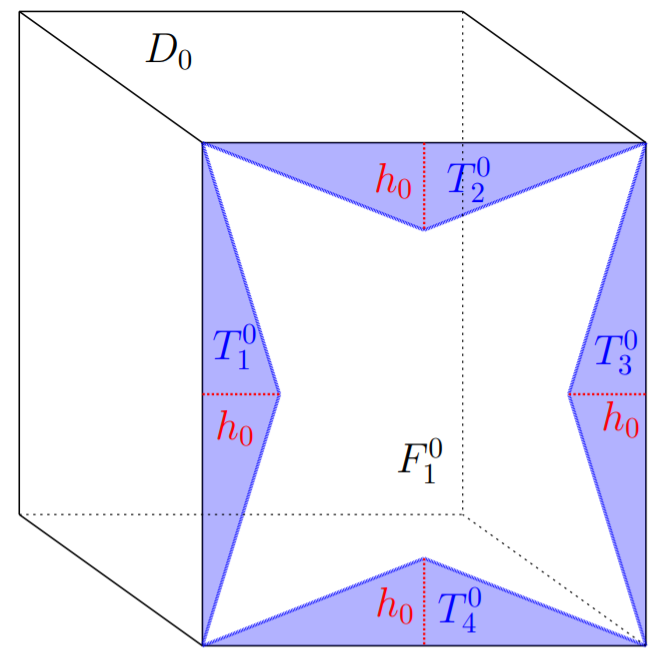}        \caption{Sketch of the construction of isosceles triangles on the face $F^0_1$ in the specific case of a cube $D_0$.}\label{fig:triangles}
\end{figure}
Thanks to the fact that $D_0,D_1\in\mathcal{D}$ and hence satisfy the same apriori assumptions, we can repeat exactly the same construction of triangles on the corresponding face $F_j^1$ of $D_1$.  We then construct a continuous piecewise affine map $\Phi_j$  defined on the $\cup_{i=1}^K T^0_i$ as follows: it is affine on each triangle of the partition and satisfies $\Phi_j(V_l^{T_i^0})=V_l^{T_i^1}-V_l^{T_i^0} $ for each $i=1,\dots K$ and $l=1,2,3$. 
By (\ref{vertices})  one has that \begin{equation}|V_l^{T_i^1}-V_l^{T_i^0}|\leq Cd_H\end{equation} for each $i=1,\dots K$ and $l=1,2,3$
and one can see that on $\cup_{i=1}^K T^0_i$ the map $\Phi_j$ satisfies
\begin{equation}\label{affine}
|\Phi_j(x)|\leq C_0 d_H,\quad \,\, |\Phi_j(x)-\Phi_j(y)|\leq C_1d_H|x-y|,\,\forall x,y\in \cup_{i=1}^K T^0_i
\end{equation}
where 
$C_0$ and $C_1$ depend only on the a-priori constants.
Consider now the  map $f$ defined on the collection of triangles $\mathcal{T}_0$ as follows: for any $x\in\mathcal{T}_0\cap F_j^0$ it satisfies $f(x)=\Phi_j(x)$. Clearly, $f$ is Lipschitz continuous and  satisfies (\ref{affine}) on $\mathcal{T}_0$.\\ 
Applying now Kirszbraun's theorem for $A=\mathcal{T}_0$  there exists a Lipschitz map $\bar{\mathcal{U}}$ from $\mathbb{R}^3$ to $\mathbb{R}^3$ which is Lipschitz continuous  and satisfies (\ref{affine}) for all $x,y\in \mathbb{R}^3$. Finally, by considering a real valued cut-off smooth function $\varphi:\mathbb{R}^3\rightarrow\mathbb{R}$ such that $0\leq \varphi\leq 1$, with compact support in $\mathcal{W}$ and with $\varphi=1$ in a tubular neighborhood of $\partial D_0$ of width $r_0/4$ and such that $|\nabla\varphi|\leq C$ with $C$ depending only on the apriori data  then it is straightforward  to see that $\mathcal{U}=\varphi\bar{\mathcal{U}}$ satisfies the desired properties \eqref{ass1:U} - \eqref{ass4:U}.

\end{proof}

As a consequence of the previous construction we have the following
\begin{proposition}\label{prop:property_Phi}
The map
\begin{equation*}
    \Phi_t=I+t\mathcal{U}\, \qquad t\in[0,1]
\end{equation*}
has the following properties
\begin{align}
    & \Phi_t \ \textrm{is piecewise affine on}\ \partial D_0;\label{eq:property_phi_1}\\
    & \Phi_t\in W^{1,\infty}(\om)\ \textrm{is invertible};\label{eq:property_phi_2}\\
    & |D\Phi_t-I|,\, |D\Phi^{-1}_t-I|\leq c td_H; \label{eq:property_phi_3} \\
    & \Phi_t(\om\setminus D_0)\subset\om; \label{eq:property_phi_4}\\
    & \left|\frac{d}{dt}\Phi_t\right|,\, \left|\frac{d}{dt}\Phi^{-1}_t\right|\leq Cd_H; \label{eq:property_phi_5} \\
    &\left|\frac{d}{dt}D\Phi_t\right|,\,\left|\frac{d}{dt}D\Phi^{-1}_t\right|\leq Cd_H; \label{eq:property_phi_6}\\
    & \left|\frac{d}{dt}D\Phi^{-1}_t+D\mathcal{U}\right|\leq Ctd^2_H; \label{eq:property_phi_7}\\
    &\left|\frac{d}{dt}(D\Phi^{-1}_t)^T+D\mathcal{U}^T\right|\leq Ctd^2_H, \label{eq:property_phi_8}
\end{align}
where $D\Phi_t$, $D\Phi^{-1}_t$, and $D\mathcal{U}$ are the Jacobian matrices of $\Phi_t$, $\Phi^{-1}_t$, and $\mathcal{U}$, respectively and $d_H$ is as in \eqref{dH}.
\end{proposition}
\begin{proof}
Property \eqref{eq:property_phi_1} follows immediately from the definition of $\mathcal{U}$. In order to prove \eqref{eq:property_phi_2}, notice that 
\begin{equation*}
    |D\Phi_t-I|=t|D\mathcal{U}|\leq t\frac{\widetilde{C}_0 d_H}{r_0}\leq t\frac{\widetilde{C}_0\delta_0}{4r_0},
\end{equation*}
where the last inequality comes from the stability estimate \eqref{eq:stability_estimate}.
Now, by the equivalent Proposition 3.4 of \cite{BerFraVes21} possibly taking $\delta_0$ small enough so that
\begin{equation*}
    \frac{\widetilde{C}_0\delta_0}{4r_0}<\frac{1}{2},
\end{equation*}
it follows that $|D\Phi_t-I|\leq 1/2$ and $\Phi_t$ is invertible for all $t\in[0,1]$. Moreover, by the Implicit Map Theorem it follows that $D\Phi^{-1}_t(y)=(D\Phi_t)^{-1}(\Phi_t^{-1}(y))$ and the analyticity in the parameter $t$ of $(D\Phi_t)^{-1}$ gives     
\begin{equation*}
|D\Phi^{-1}_t-I|\leq \frac{\widetilde{C}_0 t}{r_0}d_H.
\end{equation*}
By construction of $\Phi_t$, it holds $\Phi_t(\om\setminus D_0)\subset \om$. Estimates \eqref{eq:property_phi_5} - \eqref{eq:property_phi_8} are a consequence of \eqref{ass4:U} and the analyticity of $\Phi^{-1}_t$ and $D\Phi^{-1}_t$ with respect to $t$.
\end{proof}

\subsection{On the differentiability properties of DtN map}
In this subsection, we state some results concerning the existence of the Gateaux derivative of the local DtN map along the direction of the vector field $\mathcal{U}$ (Proposition \ref{shapederivative})  and its continuity (Proposition \ref{prop7.4}). We do not provide the proofs of the two propositions since they can be obtained in the same way as in the two-dimensional case treated in Section 5 of \cite{BerFraVes21}.  \\

Let $D_t=\Phi_t(D_0)$ and $\gamma_{D_{t}}(x)=\gamma_{D_{0}}\left( \Phi^{-1}_t(x)\right)$. 
Given $f,g\in H^{\frac{1}{2}}_{co}(\Sigma)$, let $u_t$ be the solution to \eqref{eq:main_prob} with $\gamma_{D}=\gamma_{D_{t}}$ and $v_t$ the solution of the same equation satisfied by $u_t$ but with Dirichlet boundary data $g$ (see \eqref{eq:main_prob}).

We define
\begin{equation*}
    F(t,f,g)=\bigg\langle\Lambda^{\Sigma}_{\gamma_{D_t}}f\lfloor_{\Sigma},g\bigg\rangle=\int_{\om}\gamma_{D_t}\nabla u_t \cdot \nabla v_t\, dx=\bigg\langle\frac{\partial u_t}{ \partial n}\lfloor_{\Sigma},g\bigg\rangle
\end{equation*}
and 
\begin{equation*}
    \begin{aligned}
    A(t)&= D\Phi^{-1}_t (D\Phi^{-1}_t)^T\, \textrm{det}\,( D\Phi_t)\\
    \mathcal{A}=A'(0)&=\textrm{div}\, \mathcal{U}I-(D\mathcal{U}+D\mathcal{U}^T).
    \end{aligned}
\end{equation*}
The following results hold.
\begin{proposition}\label{shapederivative}

$F(t,f,g)$ is differentiable for all $t_0\in [0,1]$ and 
\begin{equation*}
    F'(t_0,f,g)=-\int_{\om}\gamma_{D_{t_0}}\mathcal{A}_{t_0}\nabla u_{t_0} \cdot \nabla v_{t_0}\, dx
\end{equation*}
where
\begin{equation*}
    \mathcal{A}_{t_0}=\frac{d}{dt}\left(\left(D\Phi_{t_0,t}^{-1}\right)\left(D\Phi_{t_0,t}^{-1}\right)^T\det\left(D\Phi_{t_0,t}\right)\right)\Big\lfloor_{t=t_0}
\end{equation*}
and $\Phi_{t_0,t}=I+t\mathcal{U}_{t_0}$, and $\mathcal{U}_{t_0}$ is a $W^{1,\infty}(\Omega)$ map satisfying the analogue properties as those introduced for $\mathcal{U}$ with $D_{t_{0}}$ instead of $D_0$.
In particular for $t=0$ 
\begin{equation*}
   F'(0,f,g)=-\int_{\om}\gamma_{D_0}\mathcal{A}\nabla u_0\cdot \nabla v_0\, dx.
\end{equation*}
\end{proposition}


\begin{proposition}\label{prop7.4}
There exist constants $C, \beta_3>0$ depending only on the a priori data such that for all $t\in[0,1]$
\begin{equation*}
    |F'(t,f,g)-F'(0,f,g)|\leq C \|f\|_{H^{\frac{1}{2}}_{co}(\Sigma)}\|g\|_{H^{\frac{1}{2}}_{co}(\Sigma)} t^{\beta_3}d_H^{1+\beta_3},
\end{equation*}
for $d_H$ as in \eqref{dH}.
\end{proposition}
\subsection{Lower bound of the derivative} 
We now establish a lower bound for the derivative of $F$ at $t=0$. More precisely, we prove the following
\begin{proposition}\label{prop:lower_bound_F'}
There exists a constant $m_0>0$, depending only on the a priori data such that 
\begin{equation*}
\|F'(0)\|_*\geq m_0 d_H. 
\end{equation*}
where 
\begin{equation*}
    \|F'(0)\|_*=\sup\left\{\frac{|F'(0,f,g)|}{\|f\|_{H^{\frac{1}{2}}_{co}(\Sigma)}\|g\|_{H^{\frac{1}{2}}_{co}(\Sigma)}}\, :\, f,g\neq 0\right\}
\end{equation*}
and $d_H$ is given in \eqref{dH}.
\end{proposition}
Before proving the lower bound, we state the following lemma which is a special case of Proposition 1.6 in \cite{LiNir03}.
\begin{lemma}
Let $B_r$ be a ball of radius $r>0$ centered at the origin, and let $B^{\pm}_{r}$ be the upper and the lower half ball and let $\gamma_1,\ \gamma_2$ be two positive constants. Let $v\in H^1(B_r)$ be a solution to
\begin{equation}\label{eq:equation_lemma_Li_Nir}
\dive\left( (\gamma_1+(\gamma_2-\gamma_1)\chi_{B^+_r})\nabla v\right)=0,\qquad \textrm{in}\ B_r.
\end{equation}
Then $v\in C^{\infty}(\overline{B}^{\pm}_r)$ and for all $\delta>0$ there exists a constant depending only on $\gamma_1,\ \gamma_2$ and $\delta$ such that
\begin{equation}\label{eq:estimate_Li_Nir}
\|\nabla v\|_{L^{\infty}(B_{(1-\delta)r})}\leq C \|v\|_{L^2(B_r)}.
\end{equation}
\end{lemma}

\begin{proof}[Proof of Proposition \ref{prop:lower_bound_F'}]
We set
\begin{equation*}
    W=(V^{D_0}_{1}-V^{D_1}_{1}, V^{D_0}_{2}-V^{D_1}_{2},\ldots, V^{D_0}_{N}-V^{D_1}_{N}).
\end{equation*}
By Proposition \ref{distvert} we have that 
\begin{equation}\label{eq:estimate_length_W}
C^{-1}d_H\leq |W|\leq C d_H,
\end{equation}
where $C$ depends on the a priori data. We normalize by the length $|W|$ of the vector $W$ by setting
\begin{equation*}
    \widetilde{\mathcal{U}}=\frac{\mathcal{U}}{|W|},\qquad \widetilde{\mathcal{A}}=\frac{\mathcal{A}}{|W|},
\end{equation*}
and 
\begin{equation*}
H(f,g):=-\int_{\om}\gamma_{D_0}\widetilde{\mathcal{A}}\nabla u_0\cdot \nabla v_0\, dx,
\end{equation*}
so that $F'(0,f,g)=|W| H(f,g)$. Let $m_1=\|H\|_*$ the operator norm so that
\begin{equation*}
|H(f,g)|\leq m_1 \|f\|_{H^{\frac{1}{2}}_{co}(\Sigma)} \|g\|_{H^{\frac{1}{2}}_{co}(\Sigma)}, \qquad \forall f,g\in H^{\frac{1}{2}}_{co}(\Sigma).
\end{equation*}
In particular, we have 
\begin{equation}\label{eq:norm_F'}
    \|F'(0)\|_*=|W|\|H\|_*.
\end{equation}
We divide the proof in three main steps.\\

\textit{Step 1.} To start with, we choose special boundary values $f^{\sharp},\ g^{\sharp}$ by setting for $y,z\in \oms\setminus\overline{\om}$,
\begin{equation*}
f^{\sharp}(\cdot)=\Gs_0(\cdot,y)\lfloor_{\partial\om},\qquad g^{\sharp}(\cdot)=\Gs_0(\cdot,z)\lfloor_{\partial\om}
\end{equation*}
where $\Gs_0(\cdot,y),\ \Gs_0(\cdot,z)$ are the Green's functions defined in \eqref{eq:equ_green_func} with conductivity $\gamma_{D_0}$ and singularity at $y$ and $z$ respectively.  
With these choices, we consider the corresponding solutions $\Gs_0(x,y)$ and $\Gs_0(x,z)$ that we will still denote by $u_0$ and $v_0$ for the sake of brevity. Since $y,z\in \om_0\setminus\overline{\om}$, $u_0,\ v_0\in H^1(\om)$ and we can define 
\begin{equation*}
    \Theta(y,z):=-\int_{\om}\gamma_{D_0}\widetilde{\mathcal{A}}\nabla u_0 \cdot \nabla v_0\, dx=-\int_{\om}\gamma_{D_0}\widetilde{\mathcal{A}}\nabla\Gs_0(\cdot,y)\cdot \nabla \Gs_0(\cdot,z)\, dx.
\end{equation*}
First, observe that for $y,z\in B_{r_1}(P_0)\subset \om_0$, see \eqref{eq:ball},
\begin{equation*}
    \Theta(y,z)=H(\fs,\gs),
\end{equation*}
hence
\begin{equation*}
    |\Theta(y,z)|\leq m_1 \|\fs\|_{H^{\frac{1}{2}}_{co}(\Sigma)} \|\gs\|_{H^{\frac{1}{2}}_{co}(\Sigma)}.
\end{equation*}
From \eqref{eq: estimate_G_boundary}, we have that $\|\fs\|_{H^{\frac{1}{2}}_{co}(\Sigma)}, \|\gs\|_{H^{\frac{1}{2}}_{co}(\Sigma)}\leq C r_1^{-\frac{1}{2}}$. Hence,
\begin{equation}\label{eq:estimate_T_calligraphic}
    |\Theta(y,z)|\leq \frac{C m_1}{r_1},\qquad \forall y,z\in B_{r_1}(P_0).
\end{equation}
\textit{Step 2.} As second step, we use the properties of $\Theta(y,z)$ to go from the distributed formula to the boundary formula on $\partial D_0$, far from vertices and edges. To this purpose, we consider a tubular neighborhood of the edges $\{\sigma_{ij}^{D_0}\}$, with $i\neq j$, that is
\begin{equation*}
    \mathcal{B}=\bigcup_{\substack{i,j \\ i\neq j}}\bigg\{x\in\om \ : \ dist(x,\sigma_{ij}^{D_0})\leq \frac{r_0}{c_1}   \bigg\},
\end{equation*}
where $c_1>1$ depends only on the a priori data, so that 
\begin{align}
    &(\om\setminus D_0)\setminus\mathcal{B}\ \textrm{is connected};\\
    & dist(\mathcal{B},\partial\om)\geq \frac{r_0}{2}.
\end{align}
Let us write
\begin{equation*}
    \Theta(y,z)=-\int_{\om\setminus\mathcal{B}}\gamma_{D_0}\widetilde{\mathcal{A}}\nabla u_0\cdot\nabla v_0\, dx - \int_{\mathcal{B}}\gamma_{D_0}\widetilde{\mathcal{A}}\nabla u_0\cdot\nabla v_0\, dx. 
\end{equation*}
 In $\om\setminus D_0$ and in $D_0$ we have that $\widetilde{\mathcal{A}}\nabla u_0\cdot \nabla v_0=-\dive(b)$,
where 
\begin{equation*}
    b=\left(\widetilde{\mathcal{U}}\cdot \nabla u_0\right)\nabla v_0 + \left(\widetilde{\mathcal{U}}\cdot \nabla v_0\right)\nabla u_0 - \left(\nabla u_0\cdot \nabla v_0\right)\widetilde{\mathcal{U}}.
\end{equation*}
Then, we can write
\begin{equation*}
    \int_{\om\setminus\mathcal{B}}\gamma_{D_0}\widetilde{\mathcal{A}}\nabla u_0\cdot\nabla v_0\, dx=-\int_{\om\setminus(D_0\cup \mathcal{B})}\dive(b^e)\, dx- k\int_{D_0\cup \mathcal{B}}\dive(b^i)\, dx,
\end{equation*}
where we have set $b^e=b\big\lfloor_{\om\setminus D_0}$ and $b^i=b\big\lfloor_{D_0}$.
Let us now integrate by parts and denote by $\nu$ the outward unit normal vector to $\mathcal{B}$ and to $D_0$. Observing that by construction,
$\textrm{supp}\ \widetilde{\mathcal{U}}\subset \mathcal{W}$, it follows that $b=0$ on $\partial\om$. Hence, 
\begin{equation}\label{eq:div_b_ext}
    \int_{\om\setminus(D_0\cup\mathcal{B})}\dive(b^e)\, dx=-\int_{\partial\mathcal{B}\cap(\om\setminus D_0)}b^e\cdot \nu\, d\sigma(x)-\int_{\partial D_0\cap(\om\setminus\mathcal{B})}b^e\cdot \nu\, d\sigma(x),
\end{equation}
and
\begin{equation}\label{eq:div_b_in}
    \int_{D_0\cup\mathcal{B}}\dive(b^i)\, dx=\int_{\partial D_0\setminus \mathcal{B}}b^i\cdot \nu\, d\sigma(x)-\int_{\partial \mathcal{B}\cap D_0}b^i\cdot \nu\, d\sigma(x).
\end{equation}
Then by \eqref{eq:div_b_ext} and \eqref{eq:div_b_in}, it follows
\begin{equation*}
    \int_{\om\setminus\mathcal{B}}\gamma_{D_0}\widetilde{\mathcal{A}}\nabla u_0\cdot \nabla v_0\, dx=\int_{\partial\mathcal{B}}\gamma_{D_0}b\cdot \nu\, d\sigma(x) -\int_{\partial D_0\setminus \mathcal{B}}[\gamma_{D_0}b\cdot \nu]\, d\sigma(x),
\end{equation*}
where $[\cdot]$ denotes the jump along the surface $\partial D_0$. By the transmission conditions satisfied by $u_0$ and $v_0$ across $\partial D_0$ and the fact that $\widetilde{\mathcal{U}}\in W^{1,\infty}(\mathbb{R}^3)$ on $\partial D_0\setminus \mathcal{B}$, we can write
\begin{equation*}
    [\gamma_{D_0}b\cdot \nu]=\widetilde{\mathcal{U}}\cdot \nu (k-1)\mathcal{M}\nabla u_0^i\cdot \nabla v_0^i,
\end{equation*}
where $\mathcal{M}$ is the so-called polarization tensor, i.e., a $3\times3$ matrix with eigenvectors $\nu$ and $\nu^{\perp}$ and with eigenvalues $k$ and $1$. Hence, we can rewrite \eqref{eq:estimate_T_calligraphic} in the form
\begin{equation}\label{eq:T_calligraphic_v2}
    \begin{aligned}
    \Theta(y,z)=&-\int_{\mathcal{B}}\gamma_{D_0}\widetilde{\mathcal{A}}\nabla u_0\cdot \nabla v_0\, dx - \int_{\partial\mathcal{B}}\gamma_{D_0}b\cdot \nu\,dx\\ 
    &+\int_{\partial D_0\setminus \mathcal{B}}\widetilde{\mathcal{U}}\cdot \nu (k-1) \mathcal{M}\nabla u_0^i \cdot \nabla v_0^i\, dx.
    \end{aligned}
\end{equation}
\textit{Step 3.} We now use the properties of the function $\Theta$ to propagate the estimate \eqref{eq:estimate_T_calligraphic} up to points that are close to the faces of $D_0$ but far from vertices and edges.\\
From formula \eqref{eq:T_calligraphic_v2}, $\Theta(y,z)$ is well defined for $(y,z)\in\oms\setminus(D_0\cup\mathcal{B})$ and recalling that $u_0(\cdot)=\Gs_0(\cdot,y)$, $v_0(\cdot)=\Gs(\cdot,z)$, we have 
\begin{equation*}
    \dive(\gamma_{D_0}\nabla\Theta)=0,\qquad \textrm{in}\ \ \om\setminus(D_0\cup\mathcal{B})
\end{equation*}
both with respect to $y$ and $z$, i.e.
\begin{equation*}
\dive_y(\gamma_{D_0}\nabla_y\Theta(\cdot,z))=0,\quad \textrm{and}\quad \dive_z(\gamma_{D_0}\nabla_z\Theta(y,\cdot))=0, \quad \textrm{in}\ \om\setminus(D_0\cup\mathcal{B}).   
\end{equation*}
Let us now consider an arbitrary face $F_j$ of $D_0$ and let us choose $P\in F_j\cap T$ as the incenter of a triangle $T$ of $F_j\cap\mathcal{T}_0$, where $T\in \mathcal{T}_0$ and  $\mathcal{T}_0$ is the partition of triangles  defined at the beginning of the section. Consider then a ball centered at $P$ and radius 
$\frac{r_0}{2C_1}$ with $C_1=\max\left(c_1, \frac{1+\sqrt{1+\tan^2(\theta_0/2)}}{\min(1,\tan(\theta_0/2))}\right)$. 
Then by the a priori assumptions on $D_0$, $B_{\frac{r_0}{2C_1}}(P)$ is  such that it intersects $\partial D_0$ only on the face $F_j$,  $B_{\frac{r_0}{2C_1}}(P)\cap F_j$ is striclty contained in $T$, and $dist(P,\mathcal{B})\geq \frac{r_0}{2C_1}$.\\ 
Let $\widetilde{\mathfrak{c}}$ be a simple curve adjoining $P+\frac{r_0}{2C_1}\nu(P)$ with the point $\widetilde{P}_0\in B_{\frac{r_0}{2C_1}}(P_0)\subset \om_0$ such that $\widetilde{\mathfrak{c}}\subset \oms\setminus D_0$ and $dist(\widetilde{\mathfrak{c}}, D_0)\leq \frac{r_0}{2C_1}$. Let 
\begin{equation*}
    \widetilde{\mathfrak{c}}'=\widetilde{\mathfrak{c}}\cup\left\{P+t\nu(P),\ t\in\left[0,\frac{r_0}{2C_1}\right]\right\}
\end{equation*}
and
\begin{equation*}
\begin{aligned}
    \mathcal{K}&=\left\{x\in(\oms\setminus D_0)\ :\ dist(x,\widetilde{\mathfrak{c}}')< \frac{r_0}{4C_1}\right\}\\
  \mathcal{K}'&=\left\{x\in(\oms\setminus D_0)\ :\ dist(x,\widetilde{\mathfrak{c}}')< \frac{r_0}{8C_1}\right\}.
\end{aligned}
\end{equation*}
Then the function $\Theta$ solves in $\mathcal{K}$ the equations 
\begin{equation*}
\dive_y(\gamma_{D_0}\nabla_y\Theta(\cdot,z))=0,\quad \textrm{and}\quad \dive_z(\gamma_{D_0}\nabla_z\Theta(y,\cdot))=0.
\end{equation*}
Let us start to estimate $\Theta(y,z)$ for $y,z\in \mathcal{K}$ using \eqref{eq:T_calligraphic_v2}.
Since $dist(\mathcal{K},\mathcal{B})\geq \frac{r_0}{4C_1}$, we have that
\begin{equation*}
    \|\nabla u_0\|_{L^2(\mathcal{B})}\leq \|\Gs_0(\cdot,y)\|_{L^2\left(\oms\setminus B_{\frac{r_0}{4C_1}}(y)\right)}\leq C,
\end{equation*}
and analogously
\begin{equation*}
    \|\nabla v_0\|_{L^2(\mathcal{B})}\leq \|\Gs_0(\cdot,z)\|_{L^2\left(\oms\setminus B_{\frac{r_0}{4C_1}}(z)\right)}\leq C,
\end{equation*}
Hence, we have that
\begin{equation}\label{eq:equation1}
    \bigg|\int_{\mathcal{B}}\gamma_{D_0}\nabla u_0\cdot\nabla v_0\, dx\bigg|\leq C\|\nabla u_0\|_{L^2(\mathcal{B})}\|\nabla v_0\|_{L^2(\mathcal{B})}\leq C.
\end{equation}
Let us now estimate the second integral on the right-hand side of \eqref{eq:T_calligraphic_v2}. For, we consider a neighborhood of $\partial\mathcal{B}$, that is
\begin{equation}
    \mathfrak{M}=\left\{x\ :\ dist(x,\partial\mathcal{B})\leq \frac{r_0}{8C_1}\right\}
\end{equation}
and
\begin{equation}
    \mathfrak{M}'=\left\{x\ :\ dist(x,\partial\mathcal{B})\leq \frac{r_0}{16C_1}\right\}.
\end{equation}
Since $u_0$ and $v_0$ are variational solutions of equation \eqref{eq:equation_lemma_Li_Nir} in $\mathfrak{M}$, we apply the estimate \eqref{eq:estimate_Li_Nir}, getting
\begin{equation*}
    \|\nabla u_0\|_{L^{\infty}(\mathfrak{M}')},\ \|\nabla v_0\|_{L^{\infty}(\mathfrak{M}')}\leq C,
\end{equation*}
where $C$ depends only on the a priori constants. Hence,
\begin{equation}\label{eq:equation2}
    \bigg|\int_{\partial\mathcal{B}}\gamma_{D_0}b\cdot  \nu\, d\sigma(x)\bigg|\leq C.
\end{equation}
For a similar reason, for points on $\partial D_0\setminus\left(\mathcal{B}\cup B_{\frac{r_0}{2C_1}}(P)\right)$, we can bound
\begin{equation}\label{eq:equation3}
    \bigg|\int_{\partial D_0\setminus\left(\mathcal{B}\cup B_{\frac{r_0}{2C_1}}(P)\right)}\widetilde{\mathcal{U}}\cdot \nu (k-1)\mathcal{M}\nabla u_0^i\cdot\nabla v_0^i\, dx\bigg|\leq C
\end{equation}
since $y,z\in \mathcal{K}$.
Finally, let us bound 
\begin{equation}\label{eq:equation4}
    \int_{\partial D_0\cap B_{\frac{r_0}{2C_1}}(P)}\widetilde{\mathcal{U}}\cdot \nu (k-1)\mathcal{M}\nabla u_0^i\cdot\nabla v_0^i\, dx.
\end{equation}
Notice that if $y,z$ are at positive fixed distance from $\partial D_0\cap B_{\frac{r_0}{2C_1}}(P)$, then we can use again \eqref{eq:estimate_Li_Nir} to estimate \eqref{eq:equation4}. On the other hand, for points $y,z$ close to $\partial D_0\cap B_{\frac{r_0}{2C_1}}(P)$, we can use \eqref{eq:estimate_green_fundamental_solution} and the explicit formula of the fundamental solution to get 
\begin{equation}\label{eq:equation5}
   \bigg| \int_{\partial D_0\cap B_{\frac{r_0}{2C_1}}(P)}\widetilde{\mathcal{U}}\cdot \nu (k-1)\mathcal{M}\nabla u_0^i\cdot\nabla v_0^i\, dx\bigg|\leq C (\overline{d}_y \overline{d}_z)^{-1},
\end{equation}
where $\overline{d}_y=dist(y, D_0)$, $\overline{d}_z=dist(z,D_0)$. Collecting all previous estimates \eqref{eq:equation1}, \eqref{eq:equation2}, \eqref{eq:equation3} and \eqref{eq:equation4} we end up with the following bound
\begin{equation}\label{eq:estimate_T_distance}
    |\Theta(y,z)|\leq C (\overline{d}_y \overline{d}_z)^{-1},\qquad \forall y,z\in \mathcal{K}.
\end{equation}
Let us set consider the following subsets of the walkway $\mathcal{K}$
\begin{equation*}
    \begin{aligned}
    \Kd&=\left\{x\in \mathcal{K}\ :\ dist(x, D_0)\geq \frac{r_0}{32 C_1} \right\}\\
    \Kd_0&=\left\{x\in \mathcal{K}\ :\ dist(x, D_0)\geq \frac{r_0}{16 C_1} \right\}.
    \end{aligned}
\end{equation*}
Then by \eqref{eq:estimate_T_distance} and the definition of $\Kd$, the following bound holds
\begin{equation*}
    |\Theta(y,z)|\leq C,\qquad \forall(y,z)\in\Kd.
\end{equation*}
Hence, thanks to \eqref{eq:estimate_T_calligraphic}, proceeding as in \cite[Theorem 5.1]{ARRV}, we can show that 
\begin{equation*}
    \|\Theta(\cdot,z)\|_{L^{\infty}(B_{R_1}(\overline{Q}))}\leq C m_1^{\delta},\qquad \forall z\in\om_0
\end{equation*}
where $\overline{Q}=P+\frac{r_0}{4C_1}\nu(P)$, $R_1=\frac{r_0}{8C_1}$ and $\delta\in(0,1)$. Similarly, we derive
\begin{equation}\label{eq:uniform_estimate_T}
    \|\Theta(y,\cdot)\|_{L^{\infty}(B_{R_1}(\overline{Q}))}\leq C m_1^{\delta^2},\qquad \forall y\in B_{R_1}(\overline{Q}).
\end{equation}

We now apply the three spheres inequality for harmonic functions to $\Theta(\cdot,z)$ in the balls
\begin{equation*}
    B_{\overline{R}_1}(\overline{Q})\subset B_{\overline{R}_2}(\overline{Q})\subset B_{\overline{R}_3}(\overline{Q}),
\end{equation*}
for 
\begin{equation*}
    \overline{R}_1=\frac{R_1}{2},\quad \overline{R}_2=\frac{r_0}{4C_1}-\frac{r}{2},\quad \overline{R}_3=\frac{r_0}{4C_1}-\frac{r}{4},
\end{equation*}
with $r$ to be chosen. We have for $z\in B_{\overline{R}_1}(\overline{Q})$ and $\vartheta_r=\frac{\log\left( \frac{\overline{R}_3}{\overline{R}_2}\right)}{\log\left( \frac{\overline{R}_3}{\overline{R}_1}\right)}$ that
\begin{equation*}
    \|\Theta(\cdot,z)\|_{L^{\infty}\left(B_{\overline{R}_2}(\overline{Q})\right)}\leq \|\Theta(\cdot,z)\|_{L^{\infty}\left(B_{\overline{R}_1}(\overline{Q})\right)}^{\vartheta_r} \|\Theta(\cdot,z)\|_{L^{\infty}\left(B_{\overline{R}_3}(\overline{Q})\right)}^{1-\vartheta_r},
\end{equation*}
and from \eqref{eq:uniform_estimate_T} and \eqref{eq:estimate_T_distance} we find
\begin{equation*}
    \|\Theta(\cdot,z)\|_{L^{\infty}\left(B_{\overline{R}_2}(\overline{Q})\right)}\leq \left(\frac{1}{r}\right)^{1-\vartheta_r}m_2^{\vartheta_r}
\end{equation*}
where 
\begin{equation}\label{eq:aux}
m_2=C m_1^{\delta^2}.    
\end{equation}
Hence
\begin{equation*}
    |\Theta(y_r,z)|\leq C \left(\frac{1}{r}\right)^{1-\vartheta_r}m_2^{\vartheta_r} \leq \frac{m_2^{\vartheta_r}}{r},\qquad \forall z\in B_{\overline{R}_1(\overline{Q})}.
\end{equation*}
We now consider $\Theta(y_r,\cdot)$ in the same disks getting 
\begin{equation*}
    |\Theta(y_r,y_r)|\leq C \left(\frac{1}{r^2}\right)^{1-\vartheta_r}\left(\frac{1}{r}\right)^{\vartheta_r}m_2^{\vartheta_r^2} \leq \frac{m_2^{\vartheta_r^2}}{r^2}.
\end{equation*}
Hence
\begin{equation}\label{eq:estimate_above_T_yr}
    |\Theta(y_r,y_r)|\leq C\frac{m_2^{\vartheta_r^2}}{r^2}.
\end{equation}
\textit{Step 4.} We now want to estimate $\Theta(y_r,y_r)$ from below. We start from 
\begin{equation*}
\begin{aligned}
    \Theta(y_r,y_r)=&-\int_{\mathcal{B}}\gamma_{D_0}\widetilde{\mathcal{A}}\nabla u_0\cdot \nabla v_0\, dx-\int_{\partial\mathcal{B}}\gamma_{D_0}b\cdot \nu\, d\sigma(x)\\
    &+\int_{\partial D_0 \setminus\left(\mathcal{B}\cup B_{\frac{r_0}{2C_1}}(P) \right)}\widetilde{\mathcal{U}}\cdot \nu (k-1) \mathcal{M}\nabla u_0^i \cdot \nabla v_0^i\, d\sigma(x)\\
    &+\int_{\partial D_0 \cap B_{\frac{r_0}{2C_1}}(P)}\widetilde{\mathcal{U}}\cdot \nu (k-1) \mathcal{M}\nabla u_0^i \cdot \nabla v_0^i\, d\sigma(x):=I_1+I_2+I_3+I_4.
\end{aligned}
\end{equation*}
From estimates \eqref{eq:equation1}, \eqref{eq:equation2}, and \eqref{eq:equation3} we get
\begin{equation*}
    |I_i|\leq C,\qquad i=1,2,3,
\end{equation*}
where $C$ depends only on the a priori data. To evaluate $I_4$ from below, we use \eqref{eq:estimate_green_fundamental_solution}
and add and subtract $(\widetilde{\mathcal{U}}\cdot \nu)(P)$ in the integral. A straightforward computation then give for $r\leq \frac{r_0}{16C_1}$, 
\begin{equation*}
    |I_4|\geq C \frac{|(\widetilde{\mathcal{U}}\cdot \nu)(P)|}{r^2}-\frac{C}{r}.
\end{equation*}
Hence,
\begin{equation*}
    |\Theta(y_r,y_r)|\geq C \frac{|(\widetilde{\mathcal{U}}\cdot \nu)(P)|}{r^2}-\frac{C}{r},
\end{equation*}
and by \eqref{eq:estimate_above_T_yr} we finally get
\begin{equation*}
    |(\widetilde{\mathcal{U}}\cdot \nu)(P)|\leq C(m_2^{\vartheta_r^2}+r).
\end{equation*}
If $m_2\leq e^{-(16)^4}$, i.e. (see \eqref{eq:aux})
\begin{equation}\label{eq:estimate_m1}
    m_1\leq \left(\frac{e^{-(16)^4}}{C}\right)^{\frac{1}{\delta^2}}=\frac{1}{C_1}
\end{equation}
where $C_1$ depends only on the a priori data, we can pick up
\begin{equation*}
    r=\overline{r}=\frac{r_0}{C_1}|\log m_2|^{-\frac{1}{4}}
\end{equation*}
getting 
\begin{equation*}
    |(\widetilde{\mathcal{U}}\cdot \nu)(P)|\leq C |\log m_2|^{-\frac{1}{4}}
\end{equation*}
and recalling the definition of $m_2$, we find
\begin{equation}\label{eq:estimate1b}
    |(\widetilde{\mathcal{U}}\cdot \nu)(P)|\leq C \omega_0(m_1)
\end{equation}
where $\omega_0(t)$ is an increasing concave function such that $\lim\limits_{t\to 0} \omega_0(t)=0$. Note that with a similar procedure the estimate \eqref{eq:estimate1b} can be obtained for each point in a neighborhood of $P$ in the triangle $T$ containing $P$. Since $\widetilde{\mathcal{U}}$ is affine on the triangle $T$  the estimate holds also on the corresponding edge ${\sigma_{ij}^{D_0}}$ and at the adjoining vertices. We can repeat this argument for each side of the face $F_j$. Hence, if $\{V^{D_0}_{ij}\}_i$, for $1\leq i\leq N_j$, indicate the vertices on the face $F_j$ of $D_0$
\begin{equation*}
    |\widetilde{\mathcal{U}}(V^{D_0}_{ij})\cdot \nu_j|\leq \omega_0(m_1),\qquad \textrm{for all}\ 1\leq i\leq N_j
\end{equation*}
where $\nu_j$ is the unit outward normal to the face $F_j$. In particular, recalling the definition of $\widetilde{\mathcal{U}}$ on $\partial D_0$, we get
\begin{equation*}
    \bigg|\frac{(V^{D_0}_{ij}-V^{D_1}_{ij})\cdot \nu_j}{|W|}\bigg|\leq \omega_0(m_1),\qquad \textrm{for all}\ 1\leq i\leq N_j.
\end{equation*}
We can repeat this on any face $F_j$ so that
\begin{equation}\label{eq:estimate_on_vertices}
    \bigg| \frac{(V^{D_0}_{ij}-V^{D_1}_{ij})\cdot \nu_j}{|W|}\bigg| \leq \omega_0(m_1),\qquad \textrm{for all}\ 1\leq i\leq N_j,\ j\in \{1,\ldots,H\},
\end{equation}
and $\nu_j$ normal to the face $F_j$. Let
\begin{equation*}
    |V^{D_0}_{i_0j_0}-V^{D_1}_{i_0j_0}|=\max\limits_{i,j}|V^{D_0}_{ij}-V^{D_1}_{ij}|.
\end{equation*}
Then
\begin{equation*}
    \frac{|V^{D_0}_{i_0j_0}-V^{D_1}_{i_0j_0}|}{|W|}\geq \frac{1}{N},
\end{equation*}
where $N$ is the total number of vertices of $D_0$ and $D_1$, see Proposition \ref{distvert}. Moreover, since, for the a priori information, there are three linearly independent unit directions $\nu$ for which \eqref{eq:estimate_on_vertices} holds for $i=i_0$ and $j=j_0$ then it holds for every unit direction, in particular by choosing $\overline{\nu}$ parallel to $V^{D_0}_{i_0j_0}-V^{D_1}_{i_0j_0}$, we get
\begin{equation*}
    \frac{1}{N}\leq \frac{|V^{D_0}_{i_0j_0}-V^{D_1}_{i_0j_0}|}{|W|}=\frac{|(V^{D_0}_{i_0j_0}-V^{D_1}_{i_0j_0})\cdot\overline{\nu}|}{|W|}\leq \omega_0(m_1),
\end{equation*}
which gives 
\begin{equation*}
    m_1\geq \omega_0^{-1}\left(\frac{1}{N}\right),
\end{equation*}
and recalling \eqref{eq:estimate_m1} we have
\begin{equation*}
    m_1\geq \min\left(\omega_0^{-1}\left(\frac{1}{N}\right),\frac{1}{C_1}  \right).
\end{equation*}
Finally, from the estimate \eqref{eq:estimate_length_W} and \eqref{eq:norm_F'}, we get
\begin{equation*}
    \|F'(0)\|\geq m_0 d_H,
\end{equation*}
with $m_0=C^{-1}\min\left(\omega_0^{-1}\left(\frac{1}{N}\right),\frac{1}{C_1}  \right)$.
\end{proof}

\section{Lipschitz stability: proof of Theorem \ref{mainteo}}\label{sec6}
Let $\delta_0$ be as in Proposition \ref{distvert} and let $\varepsilon_0$ be such that
\begin{equation}\label{ep}
    \widetilde{\omega}(\varepsilon_0)\leq \delta_0
\end{equation}
where $\widetilde{\omega}$ is the logarithmic modulus of continuity given in Theorem \ref{stablog}.

Let us assume first that
\begin{equation*}
    \varepsilon:=\|\Lambda^{\Sigma}_{\gamma_{{D}_0}}-\Lambda^{\Sigma}_{\gamma_{{D}_1}}\|_{\star}\leq \varepsilon_0,
\end{equation*}
so that, by Theorem \ref{stablog}, 
\begin{equation*}
d_H:=d_H(\partial D_0,\partial D_1)\leq \widetilde{\omega}(\varepsilon_0)\leq \delta_0.
\end{equation*}
For $f,g\in H^{\frac{1}{2}}_{co}(\Sigma)$  the map $F(t,f,g)$ is well defined for $t\in[0,1]$.

Notice that, by definition of $F$ and by \eqref{ep}
\begin{equation}\label{8.1}
    \left|F(1,f,g)-F(0,f,g)\right|=\left|\bigg\langle\left(\Lambda^{\Sigma}_{\gamma_{D_0}}-\Lambda^{\Sigma}_{\gamma_{D_1}}\right)f\lfloor_{\Sigma},g\bigg\rangle\right|\leq \varepsilon \|f\|_{H^{\frac{1}{2}}_{co}(\Sigma)} \|g\|_{H^{\frac{1}{2}}_{co}(\Sigma)}.
\end{equation}
Let us write
\begin{eqnarray*}
    F(1,f,g)-F(0,f,g)&=&\int_0^1F'(t,f,g)dt\\
    &=&F'(0,f,g)-\int_0^1\left[F'(t,f,g)-F'(0,f,g)\right]dt,
\end{eqnarray*}
hence
\begin{equation}\label{8.2}
    \left|F(1,f,g)-F(0,f,g)\right|\geq \left|F'(0,f,g)\right|-\int_0^1\left|F'(t,f,g)-F'(0,f,g)\right|dt,
\end{equation}
By Proposition \ref{prop7.4},
\begin{equation}\label{8.3}
    \int_0^1\left|F'(t,f,g)-F'(0,f,g)\right|dt\leq \frac{Cd_H^{1+\beta_3}}{1+\beta_3}\|f\|_{H^{\frac{1}{2}}_{co}(\Sigma)} \|g\|_{H^{\frac{1}{2}}_{co}(\Sigma)} \quad \forall f,g\in H^{\frac{1}{2}}_{co}(\Sigma)
\end{equation}

and by Proposition \ref{prop:lower_bound_F'}, there exist $f_0,g_0\in H^{\frac{1}{2}}_{co}(\Sigma)$ such that 
\begin{equation}\label{8.4}
    \left|F'(0,f_0,g_0)\right|\geq \frac{m_0d_H}{2}\|f_0\|_{H^{\frac{1}{2}}_{co}(\Sigma)} \|g_0\|_{H^{\frac{1}{2}}_{co}(\Sigma)}.
\end{equation}
Hence by \eqref{8.1}, \eqref{8.2}, \eqref{8.3} (for $f=f_0$ and $g=g_0$) and by \eqref{8.4} we have that
\begin{equation}\label{8.5}
    \varepsilon\geq \left(\frac{m_0}{2}-\frac{Cd_H^{\beta_3}}{1+\beta_3}\right)d_H.
\end{equation}
Now, by Theorem \ref{stablog} there is $\varepsilon_1\leq\varepsilon_0$ depending only on the a priori data, such that, if 
\begin{equation*}
   \varepsilon:=\|\Lambda^{\Sigma}_{\gamma_{{D}_0}}-\Lambda^{\Sigma}_{\gamma_{{D}_1}}\|_{\star}\leq \varepsilon_1, 
\end{equation*}
then
\begin{equation*}
    \frac{m_0}{2}-\frac{Cd_H^{\beta_3}}{1+\beta_3}\geq \frac{m_0}{4}
\end{equation*}
and, by \eqref{8.5},
\begin{equation}\label{conc1}\varepsilon\geq \frac{m_0}{4}d_H\end{equation}

Let us now consider the case
\begin{equation}
    \|\Lambda^{\Sigma}_{\gamma_{{D}_0}}-\Lambda^{\Sigma}_{\gamma_{{D}_1}}\|_{\star} \geq\varepsilon_1,
\end{equation}
(that includes the case $\|\Lambda^{\Sigma}_{\gamma_{{D}_0}}-\Lambda^{\Sigma}_{\gamma_{{D}_1}}\|_{\star} >\varepsilon_0$).

We have
\begin{equation}\label{conc2}
    d_H\leq 2diam(\Omega)\leq 2R_0\leq \frac{2R_0}{\varepsilon_1}\|\Lambda^{\Sigma}_{\gamma_{{D}_0}}-\Lambda^{\Sigma}_{\gamma_{{D}_1}}\|_{\star}.
\end{equation}
By \eqref{conc1} and \eqref{conc2} estimate \eqref{eq:mainteo} holds for
\begin{equation*}
    C=\max\left\{\frac{m_0}{4},\frac{2R_0}{\varepsilon_1}\right\}
\end{equation*}

\section*{Funding}
E.F. and S.V. were partly funded by Research Project 201758MTR2 of the Italian Ministry of Education, University and
Research (MIUR) Prin 2017 “Direct and inverse problems for partial differential
equations: theoretical aspects and applications”.

\appendix
\section{Upper and lower bounds for $S(y,z)$.}\label{appendix:stability}
In this section, we provide the proofs of Proposition \ref{th: stability_above} and Proposition \ref{th: stability_below}.  
\begin{proof}[Proof of Proposition \ref{th: stability_above}]
We divide the proof of the proposition into four steps.\\
\textit{Step 1: for all $y,z\in B_{r_1}(P_0)$, with $P_0\in\om_0$, it holds
\begin{equation}\label{eq:rough_estimate_S}
    \big|S(y,z)\big|\leq C \varepsilon,
\end{equation}
where $C$ is a constant depending on the a priori data.
}
\begin{proof}[Proof of Step 1]
By the Alessandrini identity \eqref{eq:Alessandrini_identity} specialized to the case $u_0(\cdot)=\Gs_0(\cdot,y)$ and $u_1(\cdot)=\Gs_1(\cdot,z)$, we find
\begin{equation*}
\begin{aligned}
    \big|S(y,z)\big|&=\bigg|\big\langle (\Lambda_{\gamma_{D_0}}^{\Sigma}-\Lambda_{\gamma_{D_1}}^{\Sigma})\Gs_0(\cdot,y)\big\lfloor_{\Sigma}, \Gs_1(\cdot,z)\big\lfloor_{\Sigma} \big\rangle\bigg|\\
    &\leq \big\| \Lambda^{\Sigma}_{\gamma_{D_0}}-\Lambda^{\Sigma}_{\gamma_{D_1}}\big\|_{\star} \|\Gs_0(\cdot,y)\|_{H^1(\om)} \|\Gs_1(\cdot,z)\|_{H^1(\om)} \leq C\varepsilon
\end{aligned}    
\end{equation*}
thanks to \eqref{eq:bound_estimates_green_func} and \eqref{eq:smallness_diff_DtN_map}, where $C$ depends only on the a priori data.
\end{proof}
In the next step, we get an estimate of $S(y,z)$, when the point $y$ belongs to $\om_0$ while $z$ is in $\mathcal{G}$ but is far from edges and vertices of $\Omega_{\mathcal{G}}$ where $H^1$-estimates of the Green function do not hold.  \\
\textit{Step 2: let $y\in B_{r_1}(P_0)$, where $P_0\in\om_0$. For all $C_1>1$ and $z\in \oms\setminus \left(\Omega_{\mathcal{G}} \cup \bigcup\limits_{P_1\in \sigma^{\Omega_{\mathcal{G}}}_{ij}} B_{\frac{r_0}{C_1}}(P_1)  \right)$, with $i\neq j$, there exists a constant $C$ depending on the a priori data and $C_1$
such that
\begin{equation}\label{eq:rough_estimate_S2}
    \big|S(y,z)\big|\leq Cd_z^{-\frac{1}{2}}.
\end{equation}
where $d_z=dist(z,\partial\Omega_{\mathcal{G}})$.
}
\begin{proof}[Proof of Step 2]
Let us consider \eqref{eq:func_S}. Then
\begin{equation*}
    \big|S(y,z)\big|\leq |k-1|\Bigg\{\int_{D_0} |\nabla \Gs_0(x,y)\cdot \nabla \Gs_1(x,z)|\, dx + \int_{D_1} |\nabla \Gs_0(x,y)\cdot \nabla \Gs_1(x,z)|\, dx  \Bigg\}
\end{equation*}
hence, for $i=0,1$, we have
\begin{equation*}
\begin{aligned}
    \int_{D_i} |\nabla \Gs_0(x,y)\cdot \nabla \Gs_1(x,z)|\, dx &\leq C \|\nabla \Gs_0(\cdot,y)\|_{L^2(D_i)} \|\nabla \Gs_1(\cdot,z)\|_{L^2(D_i)} \\
    &\leq C \| \Gs_0(\cdot,y)\|_{H^1(\om^{\sharp}\setminus B_{r_1}(y))} \| \Gs_1(\cdot,z)\|_{H^1(\om^{\sharp}\setminus B_{d_z}(z))}\\
    &\leq C d_z^{-\frac{1}{2}},
\end{aligned}
\end{equation*}
where $C$ is a constant depending only on the a priori data.
\end{proof}
\textit{Step 3:  for all $y\in B_{r_1}(P_0)$, with $P_0\in \Omega_0$, it holds 
\begin{equation}\label{eq:estimate_step3}
    |S(y,\xi_h)|\leq C \frac{\varepsilon^{\eta}}{h^{\frac{1}{2}}},
\end{equation}
where 
\begin{equation*}
    \eta=\beta_2 \tau^{\Big(\frac{\big|\log \frac{h}{d_1}\big|}{|\log \chi|}+1\Big)},
\end{equation*}
and $0<\beta_2<1$ depending on the a priori data. 
}
\begin{remark}
Before proving Step 3, we note that as a consequence of Proposition \ref{lemmageometrico} is always possible to construct a path $\mathfrak{c}$ joining a point $x\in B_{r_1}(P_0)$ to a point in $\mathcal{G}$ and a tubular neighborhood of $\mathfrak{c}$, where its radius now depends also on $r_1,M_1$.
\end{remark}
\begin{remark}
In the proof of the proposition, we make an extensive use of the three spheres inequality for harmonic functions. We refer the reader to \cite{KorMey94,Kuk98,ADC} for more details. For the sake of simplicity, we recall here the statement which is adapted to our case: for every solution $w\in H^1(B_{\varrho_0}(x))$, where $B_{\varrho_0}(x)\subset \mathcal{G}$ of the equation
\begin{equation*}
    \Delta w=0\qquad \textrm{in}\,\, B_{\varrho_0}(x)
\end{equation*}
and for all $0<\varrho_1<\varrho_2<\varrho_3\leq\varrho_0$, it holds
\begin{equation}\label{eq:three-spheres}
    \|w\|_{L^{\infty}(B_{\varrho_2}(x))}\leq \|w\|^{\tau}_{L^{\infty}(B_{\varrho_1}(x))} \|w\|^{1-\tau}_{L^{\infty}(B_{\varrho_3}(x))},
\end{equation}
where $0<\tau<1$ depends on $\frac{\varrho_2}{\varrho_3}$, $\frac{\varrho_1}{\varrho_3}$.
\end{remark}
\begin{proof}[Proof of Step 3]
Thanks to Proposition \ref{lemmageometrico}, 
we use \eqref{eq:rough_estimate_S} and the three spheres inequality \eqref{eq:three-spheres} to propagate the smallness of $S(y,z)$ inside $\oms\setminus \Omega_{\mathcal{G}}$ till reaching the point $Q$. In our notation, we choose $w=S(y,\cdot)$ and $\varrho_0=R$ in \eqref{eq:three-spheres}. Then, applying once the three spheres inequality, we get
\begin{equation*}
   \|S(y,\cdot)\|_{L^{\infty}(B_{\varrho_2}(z))}\leq \|S(y,\cdot)\|^{\tau}_{L^{\infty}(B_{\varrho_1}(z))} \|S(y,\cdot)\|^{1-\tau}_{L^{\infty}(B_{\varrho_3}(z))}.
\end{equation*}
The first and second terms on the right-hand side of the previous inequality are estimated by \eqref{eq:rough_estimate_S} and \eqref{eq:rough_estimate_S2}, respectively, noticing that the worst case in \eqref{eq:rough_estimate_S2} is given by $d_z=h$, where $h$ appears in the definition of $\xi_h$. Therefore, we find
\begin{equation*}
    \|S(y,\cdot)\|_{L^{\infty}(B_{\varrho_2}(z))}\leq C \varepsilon^{\tau}\frac{1}{h^{\frac{1}{2}(1-\tau)}}
\end{equation*}
where the last inequality comes from the fact that $0<\tau<1$.
Then, we apply the three spheres inequality along a chain of balls to reach the point $Q$, that is, we get
\begin{equation*}
    \|S(y,\cdot)\|_{L^{\infty}(B_{\varrho_2}(Q))}\leq 
    C \varepsilon^{\tau^{\widetilde{\beta}_2}}\frac{1}{h^{\frac{1}{2}(1-\tau^{\widetilde{\beta}_2})}}
    \leq C \varepsilon^{\tau^{\widetilde{\beta}_2}}\frac{1}{h^{\frac{1}{2}}},
\end{equation*}
where $\widetilde{\beta_2}$ is the number of iterations of the three spheres inequality and $C$ depends on the a priori data. In order to propagate the smallness from $Q$ to $\xi_h$, we use the same procedure proposed in \cite{ADC,ADCMR}, iterating an application of the three spheres inequality \eqref{eq:three-spheres} over a chain of balls of decreasing radius and contained in a suitable cone of vertex $P$ and axis $\nu=e_3$. Finally reasoning as in \cite{ADC}, 
we find 
\begin{equation*}
    \|S(y,\cdot)\|_{L^{\infty}(B_{\rho_{k(h)}}(\xi_h))}\leq C \frac{\varepsilon^{\beta_2 \tau^{\frac{\big| \log\frac{h}{d_1}\big|}{|\log\chi|}+1}} }{h^{\frac{1}{2}} },
\end{equation*}
where $C$ depends on the a priori constant and $0<\beta_2<1$.
Hence, \eqref{eq:estimate_step3} follows.
\end{proof}
\textit{Step 4: final step.} For all $C_1>1$ and $y,z\in \oms\setminus \left(\Omega_{\mathcal{G}} \cup \bigcup\limits_{P_1\in \sigma^{\Omega_{\mathcal{G}}}_{ij}} B_{\frac{r_0}{C_1}}(P_1)  \right)$, with $i\neq j$, one can repeat the same argument as in Step 2 to get
\begin{equation*}
    |S(y,z)|\leq C (d_yd_z)^{-\frac{1}{2}},
\end{equation*}
where $C$ depends on the a priori data and on $C_1$.
In particular, choosing $y=z=\xi_h$, we find the estimate 
\begin{equation}\label{eq:estimate_S_xih}
    |S(\xi_h,\xi_h)|\leq \frac{C}{h}.
\end{equation}
Similarly as in Step 3, we can apply Proposition \ref{lemmageometrico} and an iteration of chain of balls joining a point $y\in B_{r_1}(P_0)$, where $P_0\in\oms$, to $Q$. In the application of the three spheres inequality, estimates \eqref{eq:estimate_step3} and \eqref{eq:estimate_S_xih} are now used. It holds
\begin{equation*}
    \|S(\cdot,\xi_h)\|_{L^{\infty}(B_{\varrho_2}(Q)}\leq C \frac{\varepsilon^{\widetilde{\beta}_1 \tau^{\frac{\big| \log\frac{h}{d_1}\big|}{|\log\chi|}+1}} }{h}
\end{equation*}
where $0<\widetilde{\beta}_1<1$. Finally, we apply again the three spheres inequality along a chain of balls of decreasing radius using the same construction of Step 3. Therefore, we get
\begin{equation}\label{eq:estimate_S_final}
     \|S(\cdot,\xi_h)\|_{L^{\infty}(B_{\rho_{k(h)}}(\xi_h))}\leq C \frac{\varepsilon^{\beta_1 \tau^{\frac{2\big| \log\frac{h}{d_1}\big|}{|\log\chi|}+2}} }{h}.
\end{equation}
Defining $A=1/d_1$ and $B=\frac{2}{|\log\chi|}$, we get by \eqref{eq:estimate_S_final} the estimate 
\begin{equation*}
     |S(\xi_h,\xi_h)|\leq C \frac{\varepsilon^{\beta_1 \tau^{2+B|\log A|}h^{B|\log \tau|} }}{h}.
\end{equation*}
The assertion of the theorem follows defining $C_3$ and $C_4$ as
\begin{equation*}
C_3=\beta_1 \tau^{2+2\frac{|\log A|}{|\log \chi|}},\qquad \textrm{and}\qquad C_4=2\frac{|\log \tau|}{|\log \chi|}
\end{equation*}.
\end{proof}
Next, we provide the proof of Proposition \ref{th: stability_below}.
\begin{proof}[Proof of Proposition \ref{th: stability_below}]
Let $\varrho$ be as in \eqref{eq:condition_on_varrho} and consider $\xi_h$ and $0<h<\overline{h}\varrho$, with $\overline{h}\in (0,\frac{1}{2})$ to be chosen later. By equation \eqref{eq:func_S}, we get
\begin{equation}\label{eq:I1-I2}
\begin{aligned}
    \frac{|S(\xi_h,\xi_h)|}{|k-1|}\geq\Bigg[&\bigg|\int_{D_0}\nabla \Gs_0(x,\xi_h)\cdot \nabla \Gs_1(x,\xi_h)\, dx\bigg|\\
    &-\bigg| \int_{D_1}\nabla \Gs_0(x,\xi_h)\cdot \nabla \Gs_1(x,\xi_h)\, dx\bigg|\Bigg]=:|I_1|-|I_2|.
\end{aligned}
\end{equation}
To estimate $I_2$, note that, since $\xi_h\notin\partial\Omega_{\mathcal{G}}$, we can add and subtract the gradient of the biphase fundamental solution in $I_2$, that is
\begin{equation}\label{eq:I2}
\begin{aligned}
    |I_2|=\bigg|&\int_{D_1}\left(\nabla\Gs_0(x,\xi_h)-\nabla \widehat{\Gamma}_0(x,\xi_h)\right)\cdot \left(\nabla\Gs_1(x,\xi_h)-\nabla \widehat{\Gamma}_1(x,\xi_h)\right)\, dx\\
    &+ \int_{D_1}\left(\nabla\Gs_0(x,\xi_h)-\nabla \widehat{\Gamma}_0(x,\xi_h)\right)\cdot \nabla \widehat{\Gamma}_1(x,\xi_h)\, dx\\
    &+ \int_{D_1}\nabla \widehat{\Gamma}_0(x,\xi_h)\cdot \left(\nabla\Gs_1(x,\xi_h)-\nabla \widehat{\Gamma}_1(x,\xi_h)\right)\, dx\\
    &+ \int_{D_1}\nabla \widehat{\Gamma}_0(x,\xi_h)\cdot \nabla \widehat{\Gamma}_1(x,\xi_h) \, dx\bigg|=: |I_{21}+I_{22}+I_{23}+I_{24}|.
\end{aligned}
\end{equation}
Integral $I_{21}$ can be estimated by \eqref{eq:estimate_green_fundamental_solution}, hence
\begin{equation}\label{eq:estimateI21}
    |I_{21}|\leq \int_{\oms}|\nabla\Gs_0(x,\xi_h)-\nabla \widehat{\Gamma}_0(x,\xi_h)| |\nabla\Gs_1(x,\xi_h)-\nabla \widehat{\Gamma}_1(x,\xi_h)|\, dx\leq C.
\end{equation}
Integral $I_{22}$ and $I_{23}$ can be treated analogously. For example, by Cauchy-Schwarz inequality and \eqref{eq:estimate_green_fundamental_solution}, we find that
\begin{equation}\label{eq:I22}
\begin{aligned}
|I_{22}|&\leq \int_{D_1}|\nabla\Gs_0(x,\xi_h)-\nabla \widehat{\Gamma}_0(x,\xi_h)| |\nabla \widehat{\Gamma}_1(x,\xi_h)|\, dx\\
&\leq C \|\nabla\Gs_0-\nabla \widehat{\Gamma}_0\|_{L^2(D_1)} \|\nabla\widehat{\Gamma}_1\|_{L^2(D_1)}\leq C \|\nabla\widehat{\Gamma}_1\|_{L^2(\oms\setminus B_h(\xi_h))}.   
\end{aligned}
\end{equation}
To estimate the last term in the previous inequality, we use the result in \cite[Proposition 3.4]{ADC}, that is, using the explicit behaviour of the biphase fundamental solution, that is
\begin{equation*}
    |\nabla\widehat{\Gamma}_i(x,y)|\leq \frac{C}{|x-y|^2},\qquad\forall x,y\in\mathbb{R}^3,\ x\neq y,\ i=0,1,
\end{equation*}
and spherical coordinates, it is straightforward to prove that 
\begin{equation*}
  \|\nabla\widehat{\Gamma}_1\|_{L^2(\oms\setminus B_h(\xi_h))} \leq \frac{C}{h^{\frac{1}{2}}},
\end{equation*}
hence, the use of this result in \eqref{eq:I22} gives
\begin{equation}\label{eq:estimateI22}
    |I_{22}|\leq \frac{C}{h^{\frac{1}{2}}},\qquad \textrm{and similarly}\qquad  |I_{23}|\leq \frac{C}{h^{\frac{1}{2}}}.
\end{equation}
Integral $I_{24}$ is estimated by using again the result in \cite{ADC} and the fact that $B_h(\xi_h)\subset B_{\varrho}(P)$, since $h\leq \frac{\varrho}{2}$, and moreover $D_1\cap B_{\varrho}(P)=\emptyset$. Then,
\begin{equation}\label{eq:I24}
\begin{aligned}
    |I_{24}|&\leq \int_{D_1}|\nabla \widehat{\Gamma}_0(x,\xi_h)| |\nabla \widehat{\Gamma}_1(x,\xi_h)| \, dx \leq C \int_{D_1} \frac{C}{|x-\xi_h|^4}\, dx\\
    &\leq C \int_{\mathbb{R}^3\setminus B_{\varrho}(P)} \frac{C}{|x-\xi_h|^4}\, dx.
\end{aligned}
\end{equation}
Note that $|x-\xi_h|\geq |x|-h$, hence, from the application of spherical coordinates to the last term of \eqref{eq:I24}, we find
\begin{equation*}
    |I_{24}|\leq C\int_{\varrho}^{+\infty}\frac{r^2}{(r-h)^4}\,dr,
\end{equation*}
and since $h\leq \frac{\varrho}{2}\leq \frac{r}{2}$, we have that $r-h\geq\frac{r}{2}$, hence
\begin{equation}\label{eq:estimateI24}
    |I_{24}|\leq C\int_{\varrho}^{+\infty}\frac{1}{r^2}\, dr=\frac{C}{\varrho},
\end{equation}
where $C$ depends only on the a priori data.
Finally, by estimates \eqref{eq:estimateI21}, \eqref{eq:estimateI22} and \eqref{eq:estimateI24} in \eqref{eq:I2}, we find
\begin{equation}\label{eq:estimateI2}
    |I_2|\leq C_1+\frac{C_2}{h^{\frac{1}{2}}}+\frac{C_3}{\varrho},
\end{equation}
where $C_1, C_2, C_3$ depend on the a priori data.\\
For the first integral in \eqref{eq:I1-I2}, we use the following decomposition of the domain $D_0=(D_0\cap B_{\varrho}(P))\cup (D_0\setminus B_{\varrho}(P))$, that is
\begin{equation}\label{eq:I1}
\begin{aligned}
    |I_1|\geq& \bigg| \int_{D_0\cap B_{\varrho}(P)} \nabla \Gs_0(x,\xi_h)\cdot \nabla \Gs_1(x,\xi_h)\, dx\bigg| +\\
    &-  \bigg| \int_{D_0\setminus B_{\varrho}(P)} \nabla \Gs_0(x,\xi_h)\cdot \nabla \Gs_1(x,\xi_h)\, dx\bigg|=: |I_{11}| - |I_{12}|
\end{aligned}
\end{equation}
The term $I_{12}$ can be estimated using the same procedure adopted for $I_2$, hence
\begin{equation}\label{eq:estimateI12}
    |I_{12}|\leq C_1+\frac{C_2}{h^{\frac{1}{2}}}+\frac{C_3}{\varrho}.
\end{equation}
In $I_{11}$ we add and subtract the gradient of the biphase fundamental solution $\widehat{\Gamma}_1$, that is
\begin{equation*}
\begin{aligned}
    |I_{11}|&\geq \bigg| \int_{D_0\cap B_{\varrho}(P)}\nabla \widehat{\Gamma}_0(x,\xi_h)\cdot \nabla \widehat{\Gamma}_1(x,\xi_h)\, dx\bigg| +\\
    &- \bigg| \int_{D_0\cap B_{\varrho}(P)}\left[\nabla \Gs_0(x,\xi_h)-\nabla\widehat{\Gamma}_0(x,\xi_h)\right]\cdot \nabla \Gs_1(x,\xi_h)\, dx\bigg|=:|I_{111}|-|I_{112}|.
\end{aligned}
\end{equation*}
For the estimation of $I_{112}$ we use \eqref{eq:H1_estim_green_func} and \eqref{eq:unif_est_green_fund}, that is
\begin{equation}\label{eq:estimateI112}
    |I_{112}| \leq C \int_{\oms\setminus B_h(\xi_h)}|\nabla\Gs_1(x,\xi_h)|\,dx \leq \frac{C}{h^{\frac{1}{2}}}.
\end{equation}
In the term $I_{111}$, we add and subtract the gradient of the biphase fundamental solution $\widehat{\Gamma}_1$, that is
\begin{equation*}
\begin{aligned}
    |I_{111}|\geq& \bigg| \int_{D_0\cap B_{\varrho}(P)}\nabla\widehat{\Gamma}_0(x,\xi_h)\cdot \nabla\widehat{\Gamma}_1(x,\xi_h)\, dx\bigg| +\\
    &-\bigg| \int_{D_0\cap B_{\varrho}(P)} \nabla\widehat{\Gamma}_0(x,\xi_h)\cdot\left[\nabla\Gs_1(x,\xi_h)-\nabla\widehat{\Gamma}_1(x,\xi_h)\right]\,dx\bigg|\\
    &=:|I_{1111}|-|I_{1112}|.
\end{aligned}
\end{equation*}
For the term $I_{1112}$, we use similar arguments adopted in the previous calculations and \eqref{eq:estimate_green_fundamental_solution}, hence
\begin{equation}\label{eq:estimateI1112}
    |I_{1112}| \leq \frac{C}{h^{\frac{1}{2}}}.
\end{equation}
Finally, from the results in \cite{AleVes05,BerFra11}, we have that
\begin{equation}\label{eq:estimateI1111}
    |I_{1111}|\geq \frac{C}{h}.
\end{equation}
From \eqref{eq:I1}, by estimates \eqref{eq:estimateI1111}, \eqref{eq:estimateI1112}, \eqref{eq:estimateI112} and \eqref{eq:estimateI12}, we find
\begin{equation}\label{eq:estimateI1}
    |I_1|\geq \frac{C}{h}-C_1-\frac{C_2}{h^{\frac{1}{2}}}-\frac{C_3}{\varrho}.
\end{equation}
Finally, using \eqref{eq:estimateI1} and \eqref{eq:estimateI2} into \eqref{eq:I1-I2}, we get
\begin{equation*}
    |S(\xi_h,\xi_h)|\geq \frac{C}{h}\left(1-C_1 h^{\frac{1}{2}}-\frac{C_2h}{\varrho}-C_3h \right),
\end{equation*}
where the constants $C, C_1, C_2, C_3$ depend on the a priori data. 
Therefore, there exists $\overline{h}>0$ such that, for any $0<h<\overline{h}\varrho$, the estimate \eqref{eq:estimate_from_below} follows.
\end{proof}

\bibliography{references}{}

\begin{thebibliography}{10}

\bibitem{Alberti2020}
G.~S. Alberti, Á. Arroyo, and M.~Santacesaria.
\newblock Inverse problems on low-dimensional manifolds, 2020.

\bibitem{AS}
G.~S. Alberti and M.~Santacesaria.
\newblock Calder\'{o}n's inverse problem with a finite number of measurements.
\newblock {\em Forum Math. Sigma}, 7:Paper No. e35, 20, 2019.

\bibitem{AS2021}
Giovanni~S. Alberti and Matteo Santacesaria.
\newblock Infinite-dimensional inverse problems with finite measurements.
\newblock {\em Arch. Ration. Mech. Anal.}, 243(1):1--31, 2022.

\bibitem{AlbLauStu20}
Y.~F. Albuquerque, A.~Laurain, and K.~Sturm.
\newblock A shape optimization approach for electrical impedance tomography
  with point measurements.
\newblock {\em Inverse Problems}, 36(9):095006, 27, 2020.

\bibitem{ABV}
G.~Alessandrini, E.~Beretta, and S.~Vessella.
\newblock Determining linear cracks by boundary measurements: {L}ipschitz
  stability.
\newblock {\em SIAM J. Math. Anal.}, 27(2):361--375, 1996.

\bibitem{AldeHG}
G.~Alessandrini, M.~V. de~Hoop, and R.~Gaburro.
\newblock Uniqueness for the electrostatic inverse boundary value problem with
  piecewise constant anisotropic conductivities.
\newblock {\em Inverse Problems}, 33(12):125013, 24, 2017.

\bibitem{AldeHGS}
G.~Alessandrini, M.~V. de~Hoop, R.~Gaburro, and E.~Sincich.
\newblock Lipschitz stability for the electrostatic inverse boundary value
  problem with piecewise linear conductivities.
\newblock {\em J. Math. Pures Appl. (9)}, 107(5):638--664, 2017.

\bibitem{ADC}
G.~Alessandrini and M.~Di~Cristo.
\newblock Stable determination of an inclusion by boundary measurements.
\newblock {\em SIAM J. Math. Anal.}, 37(1):200--217, 2005.

\bibitem{ADCMR}
G.~Alessandrini, M.~Di~Cristo, A.~Morassi, and E.~Rosset.
\newblock Stable determination of an inclusion in an elastic body by boundary
  measurements.
\newblock {\em SIAM J. Math. Anal.}, 46(4):2692--2729, 2014.

\bibitem{AleKim12}
G.~Alessandrini and K.~Kim.
\newblock Single-logarithmic stability for the {C}alder\'{o}n problem with
  local data.
\newblock {\em J. Inverse Ill-Posed Probl.}, 20(4):389--400, 2012.

\bibitem{ARRV}
G.~Alessandrini, L.~Rondi, E.~Rosset, and S.~Vessella.
\newblock The stability for the {C}auchy problem for elliptic equations.
\newblock {\em Inverse Problems}, 25(12):123004, 47, 2009.

\bibitem{AS13}
G.~Alessandrini and E.~Sincich.
\newblock Cracks with impedance; stable determination from boundary data.
\newblock {\em Indiana Univ. Math. J.}, 62(3):947--989, 2013.

\bibitem{AleVes05}
G.~Alessandrini and S.~Vessella.
\newblock Lipschitz stability for the inverse conductivity problem.
\newblock {\em Adv. in Appl. Math.}, 35(2):207--241, 2005.

\bibitem{BacV}
V.~Bacchelli and S.~Vessella.
\newblock Lipschitz stability for a stationary 2{D} inverse problem with
  unknown polygonal boundary.
\newblock {\em Inverse Problems}, 22(5):1627--1658, 2006.

\bibitem{BFK}
B.~Barcel\'{o}, E.~Fabes, and Jin~K. Seo.
\newblock The inverse conductivity problem with one measurement: uniqueness for
  convex polyhedra.
\newblock {\em Proc. Amer. Math. Soc.}, 122(1):183--189, 1994.

\bibitem{BerdeHFrVes}
E.~Beretta, M.~V. de~Hoop, E.~Francini, and S.~Vessella.
\newblock Stable determination of polyhedral interfaces from boundary data for
  the {H}elmholtz equation.
\newblock {\em Comm. Partial Differential Equations}, 40(7):1365--1392, 2015.

\bibitem{BerFra11}
E.~Beretta and E.~Francini.
\newblock Lipschitz stability for the electrical impedance tomography problem:
  the complex case.
\newblock {\em Comm. Partial Differential Equations}, 36(10):1723--1749, 2011.

\bibitem{BerFraVes21}
E.~Beretta, E.~Francini, and S.~Vessella.
\newblock Lipschitz stable determination of polygonal conductivity inclusions
  in a two-dimensional layered medium from the {D}irichlet-to-{N}eumann map.
\newblock {\em SIAM J. Math. Anal.}, 53(4):4303--4327, 2021.

\bibitem{BerMicPerSan18}
E.~Beretta, S.~Micheletti, S.~Perotto, and M.~Santacesaria.
\newblock Reconstruction of a piecewise constant conductivity on a polygonal
  partition via shape optimization in {EIT}.
\newblock {\em J. Comput. Phys.}, 353:264--280, 2018.

\bibitem{BerFra20}
Elena Beretta and Elisa Francini.
\newblock Global {L}ipschitz stability estimates for polygonal conductivity
  inclusions from boundary measurements.
\newblock {\em Appl. Anal.}, 101(10):3536--3549, 2022.

\bibitem{BCFLM}
A.~Borsic, C.~Comina, S.~Foti, R.~Lancellotta, and G.~Musso.
\newblock Imaging heterogeneities with electrical impedance tomography:
  laboratory results.
\newblock {\em Geotechnique}, 55(7), 2005.

\bibitem{deHoopetal}
M.~V. de~Hoop, L.~Qiu, and O.~Scherzer.
\newblock An analysis of a multi-level projected steepest descent iteration for
  nonlinear inverse problems in {B}anach spaces subject to stability
  constraints.
\newblock {\em Numer. Math.}, 129(1):127--148, 2015.

\bibitem{ErnGuer21}
Alexandre Ern and Jean-Luc Guermond.
\newblock {\em Finite elements {I}---{A}pproximation and interpolation},
  volume~72 of {\em Texts in Applied Mathematics}.
\newblock Springer, Cham, [2021] \copyright 2021.

\bibitem{FepAllBorCorDap19}
F.~Feppon, G.~Allaire, F.~Bordeu, J.~Cortial, and C.~Dapogny.
\newblock Shape optimization of a coupled thermal fluid-structure problem in a
  level set mesh evolution framework.
\newblock {\em SeMA J.}, 76(3):413--458, 2019.

\bibitem{FGS}
S.~Foschiatti, R.~Gaburro, and E.~Sincich.
\newblock Stability for the {C}alder\'{o}n's problem for a class of anisotropic
  conductivities via an ad hoc misfit functional.
\newblock {\em Inverse Problems}, 37(12):Paper No. 125007, 34, 2021.

\bibitem{F}
Farquharson~C. G.
\newblock Constructing piecewise-constant models in multidimensional
  minimum-structure inversions.
\newblock {\em Geophysics}, 73(1), 2007.

\bibitem{GS}
R.~Gaburro and E.~Sincich.
\newblock Lipschitz stability for the inverse conductivity problem for a
  conformal class of anisotropic conductivities.
\newblock {\em Inverse Problems}, 31(1):015008, 26, 2015.

\bibitem{H1}
B.~Harrach.
\newblock Uniqueness and {L}ipschitz stability in electrical impedance
  tomography with finitely many electrodes.
\newblock {\em Inverse Problems}, 35(2):024005, 19, 2019.

\bibitem{H2}
B.~Harrach.
\newblock Uniqueness, stability and global convergence for a discrete inverse
  elliptic {R}obin transmission problem.
\newblock {\em Numer. Math.}, 147(1):29--70, 2021.

\bibitem{KorMey94}
J.~Korevaar and J.~L.~H. Meyers.
\newblock Logarithmic convexity for supremum norms of harmonic functions.
\newblock {\em Bull. London Math. Soc.}, 26(4):353--362, 1994.

\bibitem{Kuk98}
I.~Kukavica.
\newblock Quantitative uniqueness for second-order elliptic operators.
\newblock {\em Duke Math. J.}, 91(2):225--240, 1998.

\bibitem{Lau18}
A.~Laurain.
\newblock A level set-based structural optimization code using {FE}ni{CS}.
\newblock {\em Struct. Multidiscip. Optim.}, 58(3):1311--1334, 2018.

\bibitem{Lau20}
A.~Laurain.
\newblock Distributed and boundary expressions of first and second order shape
  derivatives in nonsmooth domains.
\newblock {\em J. Math. Pures Appl. (9)}, 134:328--368, 2020.

\bibitem{LauMef16}
A.~Laurain and H.~Meftahi.
\newblock Shape and parameter reconstruction for the {R}obin transmission
  inverse problem.
\newblock {\em J. Inverse Ill-Posed Probl.}, 24(6):643--662, 2016.

\bibitem{LauStu16}
A.~Laurain and K.~Sturm.
\newblock Distributed shape derivative {\it via} averaged adjoint method and
  applications.
\newblock {\em ESAIM Math. Model. Numer. Anal.}, 50(4):1241--1267, 2016.

\bibitem{LiNir03}
Y.~Li and L.~Nirenberg.
\newblock Estimates for elliptic systems from composite material.
\newblock {\em Comm. Pure Appl. Math.}, 56(7):892--925, 2003.
\newblock Dedicated to the memory of J\"{u}rgen K. Moser.

\bibitem{LiuTsou}
H.~Liu and C.-H. Tsou.
\newblock Stable determination of polygonal inclusions in {C}alder\'{o}n's
  problem by a single partial boundary measurement.
\newblock {\em Inverse Problems}, 36(8):085010, 23, 2020.

\bibitem{LiuTsoYan21}
Hongyu Liu, Chun-Hsiang Tsou, and Wei Yang.
\newblock On {C}alder\'{o}n's inverse inclusion problem with smooth shapes by a
  single partial boundary measurement.
\newblock {\em Inverse Problems}, 37(5):Paper No. 055005, 18, 2021.

\bibitem{R08}
L.~Rondi.
\newblock Stable determination of sound-soft polyhedral scatterers by a single
  measurement.
\newblock {\em Indiana Univ. Math. J.}, 57(3):1377--1408, 2008.

\bibitem{Sch69}
J.~T. Schwartz.
\newblock {\em Nonlinear functional analysis}.
\newblock Notes on Mathematics and its Applications. Gordon and Breach Science
  Publishers, New York-London-Paris, 1969.
\newblock Notes by H. Fattorini, R. Nirenberg and H. Porta, with an additional
  chapter by Hermann Karcher.

\bibitem{Shi}
J.~Shi, E.~Beretta, M.V. de~Hoop, E.~Francini, and S.~Vessella.
\newblock A numerical study of multi-parameter full waveform inversion with
  iterative regularization using multi-frequency vibroseis data.
\newblock {\em Comput. Geosci.}, 24(1):89--107, 2020.

\bibitem{Zhdanov1994TheGM}
M.~S. Zhdanov and G.~V. Keller.
\newblock {\em The Geoelectrical Methods in Geophysical Exploration}.
\newblock Methods in geochemistry and geophysics. Elsevier, 1994.

\end{thebibliography}
\bibliographystyle{plain}
\end{document}